\documentclass{amsart}
\usepackage{amsmath,amsfonts,amssymb,amsthm,amsbsy,pstricks,enumerate,graphics,
pst-plot,mathtools,verbatim}

\title{Rees algebras and almost linearly presented Ideals}
\author{Jacob A. Boswell and Vivek Mukundan}
\address{Department of Mathematics, Purdue University, West Lafayette, IN 47907, USA}
\email{jaboswel@purdue.edu,vmukunda@purdue.edu}

\input xy
\xyoption{all}
\newtheorem{thm}{Theorem}
\newtheorem{cor}[thm]{Corollary}
\newtheorem{lem}[thm]{Lemma}
\newtheorem{prop}[thm]{Proposition}
\newtheorem{obs}[thm]{Observation}

\theoremstyle{definition}
\newtheorem{definition}[thm]{Definition}
\newtheorem{ex}[thm]{Example}

\newtheorem{rmk}[thm]{Remark}

\newtheorem*{rmkk}{Remark}
\newtheorem{set}[thm]{Setting}
\newtheorem{notation}[thm]{Notation}

\theoremstyle{remark}

\newcommand{\hght}{\mathrm{ht~}}
\newcommand{\coker}{\mathrm{coker~}}
\newcommand{\dep}{\mathrm{depth~}}
\newcommand{\gr}{\mathrm{grade~}}
\newcommand{\rk}{\mathrm{rank~}}
\newcommand{\reg}{\mathrm{reg}}
\newcommand{\rt}{\mathrm{rt}}
\newcommand{\sym}{\mathrm{Sym}}
\newcommand\scalemath[2]{\scalebox{#1}{\mbox{\ensuremath{\displaystyle #2}}}}
\newcommand\numberthis{\addtocounter{equation}{1}\tag{\theequation}}
\numberwithin{equation}{section}
\numberwithin{thm}{section}

   \newtheoremstyle{TheoremNum}
   {\topsep}{\topsep}              
   {\itshape}                      
   {}                              
   {\bfseries}                     
   {.}                             
   { }                             
   {\thmname{#1}\thmnote{ \bfseries #3}}
   \theoremstyle{TheoremNum}

\begin{document}

\begin{abstract}
	Consider a grade 2 perfect ideal $I$ in $R=k[x_1,\cdots,x_d]$ which is generated by forms of the same degree. Assume that the presentation matrix $\varphi$ is almost linear, that is, all but the last column of $\varphi$ consist of entries which are linear. For such ideals, we find explicit forms of the defining ideal of the Rees algebra $\mathcal{R}(I)$. We also introduce the notion of iterated Jacobian duals.
\end{abstract}
\maketitle
\section{Introduction}
In this paper we study the defining ideal of the Rees algebra of ideals in a commutative ring. The Rees algebra $\mathcal{R}(I)$ of an ideal $I$ in a commutative ring $R$ is defined to be $\mathcal{R}(I)=R[It]=R\oplus It\oplus I^2t^2\oplus\cdots$. The defining ideal of the Rees algebra is the kernel $\mathcal{A}$ of an epimorphism $\Psi :R[T_1,\dots, T_m]\rightarrow \mathcal{R}(I)$ given by $\Psi(T_i)=\alpha_it$, where $I=(\alpha_1,\dots, \alpha_m)$. Rees algebras provide an algebraic realization for the concept of blowing up a variety along a subvariety. The search for the implicit equations defining the Rees algebra is a classical and fundamental problem which has been studied for many decades. Some of the results in this direction include \cite{Vas1,HSV1,MU,MJ1,HoSV,CHW,LN1,CD,BCS,BC1,KPU1,BJ1,LP1}.

An important object in the study of Rees algebras is the symmetric algebra. The symmetric algebra $\sym(I)$ of an ideal $I$ has a presentation $$\sym(I)\cong R[T_1,\dots,T_m]/\mathcal{L},$$ where $\mathcal{L}=([T_1\cdots T_m]\cdot\varphi)$ and $\varphi$ is a presentation matrix of $I$. The map $\Psi$ above factors through the symmetric algebra. So it is enough to study the kernel of the map $\sym(I)\rightarrow \mathcal{R}(I)$. Traditionally, techniques for computing the defining ideal of $\mathcal{R}(I)$ often revolved around the notion of \textit{Jacobian dual}. For a commutative ring $R$ and an ideal $I$ with a presentation $R^s\xrightarrow{\varphi} R^m\rightarrow I\rightarrow 0$, the \textit{Jacobian dual} of $\varphi$ is defined to be a matrix $B(\varphi)$ with linear entries in $R[T_1,\dots, T_m]$ such that
\begin{equation}\label{IntroJacobianDual}
[T_1\cdots T_m]\cdot\varphi=[a_1\cdots a_r]\cdot B(\varphi)\mathrm{,~where~}I_1(\varphi)\subseteq (a_1,\dots,a_r).
\end{equation}

In the literature, the defining ideal of Rees algebras have been studied in great detail for many classes of ideals. For example, ideals generated by regular sequences (or \textit{d}-sequences,\cite{HSV1}) and grade 2 perfect ideals with linear presentation (\cite{MU}). We restrict our study to the case where $R=k[x_1,\dots,x_d]$ and $I=(\alpha_1,\dots ,\alpha_m)$ is a grade two perfect ideal minimally generated by homogeneous elements of the same degree. Using the Hilbert-Burch theorem, such an ideal can be realized as the ideal generated by the maximal minors of a $m\times m-1$ matrix with homogeneous entries of constant degree along each column. We further restrict the presentation matrix $\varphi$ of $I$ to be \textit{almost linearly presented}, that is, all but the last column of $\varphi$ are linear and the last column consists of homogeneous entries of arbitrary degree $n\geq 1$. The $G_d$ condition is also an important ingredient in the study of $\mathcal{A}$. Here, the $G_d$ condition means that $\mu(I_p)\leq \mathrm{ht~}p~\mathrm{~for~every~} p\in V(I)~\mathrm{~with~}\mathrm{ht~}p\leq d-1$. An earlier study of Rees algebra of ideals of this type, when $d=2$, was done by Kustin, Polini and Ulrich in \cite{KPU1}. We generalize their results for $d>2$ and also present another form of the defining ideal of $\mathcal{R}(I)$.

The $G_d$ condition forces some power of the ideal $(\underline{x})$ to annihilate the kernel of the map $\sym(I)\rightarrow \mathcal{R}(I)$ i.e, $\mathcal{A}=\mathcal{L}:(\underline{x})^\infty$. Since $\mathrm{dim~}\mathcal{R}(I)=d+1$, notice that $\mathcal{A}=\mathcal{L}:(\underline{x})^\infty$ is a prime ideal of height $m-1$.

One of the recurring features of the proofs in this paper is that the ideal $I_d(B(\varphi'))$ attains the maximum possible height, namely $m-d-1$ (where $\varphi'$ is a matrix obtained from $\varphi$ by removing the last column). This led us to study $\mathcal{L}:(\underline{x})^\infty$ in a more general setting. \vskip 2mm
\begin{thm}\label{introgoup}
Let $R$ be a Cohen Macaulay local ring containing a field $k$ and $\underline{a} = a_1, \ldots, a_r$ an $R$-regular sequence with $r>0$. Let $S = R[T_1, \ldots, T_m]$ with $T_1, \ldots, T_m$ indeterminates over $R$ and $\psi$ be an $r \times s$ matrix with entries in $k[T_1, \ldots, T_m]$ so that each column consists of homogeneous elements of the same positive degree. If $(\underline{a} \cdot \psi) :_S (\underline{a})^\infty$ is a prime ideal of height $s$, then $I_r(\psi)$ is a prime ideal of $k[T_1, \ldots, T_m]$ of height $\mathrm{max}\{0, s-r+1\}$ and $(\underline{a}\cdot \psi) :_S (\underline{a})$ is a geometric residual intersection. Furthermore, \[(\underline{a}\cdot \psi) :_S (\underline{a})^\infty = (\underline{a}\cdot \psi) :_S (\underline{a}) = (\underline{a}\cdot \psi) + I_r(\psi).\]
	\end{thm}
\vskip 1mm
As a consequence, in $R=k[x_1,\dots,x_d]$, consider  an $m\times s$ matrix $\phi$ with homogeneous entries of constant degree along each column and $\phi'$ an $m\times s'$ submatrix of $\phi$ consisting of columns of $\phi$ whose entries are linear. Now if $(\underline{x}\cdot B(\phi)):(\underline{x})^\infty$ is a prime ideal of height $s$ in $R[T_1,\dots,T_m]$, then Theorem \ref{introgoup} allows us to prove that   $(\underline{x}\cdot B(\phi')):(\underline{x})^\infty=(\underline{x}\cdot B(\phi')):(\underline{x})=(\underline{x}\cdot B(\phi'))+I_d(B(\phi'))$  is a prime ideal of of height $s'$ (we refer to Section 2 for the details).

Applying this result to the case of almost linearly presented grade 2 perfect ideals satisfying the $G_d$ condition in $R=k[x_1,\dots,x_d]$, we prove that $\mathcal{L}:(\underline{x})=\mathcal{L}+I_d(B(\varphi))$. Furthermore $I_d(B(\varphi'))$ attains maximum height. We show that, for the above type of ideals, $\mathcal{A}=\mathcal{L}:(\underline{x})^n$, which constitutes the first form of the defining ideal of $\mathcal{R}(I)$ we obtain in this paper. Recall that $n$ is the degree of the entries in the last column of the presentation matrix $\varphi$. \vskip 2mm
\begin{thm}\label{introcolonform}
	Let $R=k[x_1,\dots,x_d]$ and $I$ be a grade 2 perfect ideal satisfying the $G_d$ condition. If $I$ is almost linearly presented, that is, all but the last column of the presentation matrix are linear and the last column consist of homogeneous entries of arbitrary degree $n\geq 1$, then the defining ideal of $\mathcal{R}(I)$ sastisfy $\mathcal{A}=\mathcal{L}:(\underline{x})^n$
\end{thm}
\vskip 1mm
This form is computationally inexpensive compared to $\mathcal{L}:(\underline{x})^\infty$. Notice that when $n=1$ the presentation matrix $\varphi$ is linear. Theorem \ref{introcolonform} now shows that $\mathcal{A}=\mathcal{L}:(\underline{x})=\mathcal{L}+I_d(B(\varphi))$. The equality $\mathcal{A}=\mathcal{L}+I_d(B(\varphi))$ is known as the \textit{expected form} of the defining ideal of the Rees algebra. This recovers the result proved by Morey, Ulrich in \cite{MU}.


The second form of the defining ideal is obtained by following work of Kustin, Polini, Ulrich in \cite{KPU1}. Among other results, they characterize the defining ideal of the Rees algebra of such ideals in $k[x_1,x_2]$, but, some of the techniques presented in this paper resisted generalization to $k[x_1,\dots,x_d]$. One of the most glaring deficiencies in the case of $k[x_1,\dots, x_d]$ is the lack of a characterization of the presentation matrix $\varphi$ similar to \cite[2.1]{KPU1}. We first construct a Cohen-Macaulay ring $A$ in which the defining ideal of the Rees algebra, $\overline{\mathcal{A}}$, is a height one prime ideal. The ring $A$ being a Cohen-Macaulay domain is attributed to the fact that $I_d(B(\varphi'))$ has maximal height. Now we construct an ideal $K$ in $R[T_1,\dots,T_m]$ such that $\overline{K}$ is an height one unmixed ideal in $A$ satisfying $\overline{\mathcal{A}}\cong_A \overline{K}^{(n)}$. \vskip 2mm
\begin{thm}\label{intro2ndform}
Let $R=k[x_1,\dots, x_d]$ and $I$ be a grade 2 perfect ideal. Suppose that the presentation matrix $\varphi$ of $I$ is almost linear, that is, all but the last column of $\varphi$ are linear and the last column consist of homogeneous entries of degree $n$. Also assume that $I$ satisfies the $G_d$ condition and $\mu(I)>d$. Let $\varphi '$ denote the linear matrix obtained from $\varphi$ by removing the last column \textup{(}the non-linear column\textup{)}. Let $A=R[T_1,\dots, T_m]/([x_1\cdots x_d]\cdot B(\varphi'),I_d(B(\varphi')))$ and 
$$K=(x_d)+([x_1\cdots x_{d-1}]\cdot B)+I_{d-1}(B)$$
where $B$ is obtained from $B(\varphi')$ by removing the last row of $B(\varphi')$. Then in the ring $A$, the defining ideal of $\mathcal{R}(I)$ satisfies $\overline{\mathcal{A}}\cong_A \overline{K}^{(n)}(-1)$.
\end{thm}
\vskip 1mm
This characterization is a generalized version of \cite[1.11]{KPU1}. However, it is hard to find an explicit generating set similar to \cite[3.6]{KPU1}. Using the construction in Theorem \ref{intro2ndform}, we can also show that $n$ is the smallest possible integer for a description as in Theorem \ref{introcolonform}.

A search for an explicit generating set of $\mathcal{A}=\mathcal{L}:(x_1,\dots, x_d)^n$, led to the concept of \textit{iterated Jacobian duals}. This method extends the notion of Jacobian duals, and helps in constructing generators for $\mathcal{A}$. Again, we study this procedure in a more general setting.  For any $m\times s$ matrix $\phi$ in a Noetherian ring with $I_1(\phi)\subseteq (a_1,\cdots,a_r)$, we set $B_1(\phi)=B(\phi)$ (\mbox{see }\eqref{IntroJacobianDual}) and we iteratively construct $B_i(\phi)$ from $B_{i-1}(\phi)$ (we refer to Section 4 for details on the construction). Let $\mathcal{L}$ denote the ideal defining $\sym(\coker \phi)$. By construction, $\mathcal{L}+I_r(B_i(\phi))\subset \mathcal{L}+I_r(B_{i+1}(\phi))$. Though $B_i(\phi)$ may not be unique, we exhibit the unique nature of $\mathcal{L}+I_r(B_i(\phi))$ when $a_1,\dots,a_r$ is an $R$-regular sequence.\vskip 2mm
\begin{thm}
	Let $R$ be a Noetherian ring and $\phi$ be a $m\times s$ matrix with entries in $R$. Suppose $I_1(\phi)\subseteq (a_1,\dots,a_r)$ and $a_1,\dots,a_r$ is a regular sequence.  Then the ideal $\mathcal{L}+I_r(B_i(\phi))$ of $R[T_1,\dots,T_m]$ is uniquely determined by the matrix $\phi$ and the regular sequence $a_1,\dots,a_r$.
\end{thm}
\vskip 1mm
 The procedure of iterated Jacobian duals, was independently studied, in the case of $k[x_1,x_2]$ when $m=3$ by Hong, Simis and Vasoncelos in \cite{HoSV} and Cox, Hoffman and Wang in \cite{CHW}. The construction presented in these papers are slightly different from ours.

In the context of almost linearly presented grade 2 perfect ideals in \\$R=k[x_1,\dots,x_d]$ satisfying the $G_d$ condition, we can show that $I_1(\varphi)=(\underline{x})$. Further we show that $(\mathcal{L},I_d(B_i(\varphi)))\subset \mathcal{L}:(\underline{x})^i$, and hence $(\mathcal{L}+I_d(B_i(\varphi)))\subset\mathcal{A}$.  Referring to the generating set of the defining equations of $\mathcal{R}(I)$ presented in \cite[3.6]{KPU1}, for $d=2$ we observe that the defining ideal $\mathcal{A}$ is not always equal to an ideal of an iterated Jacobian dual. But we present a condition, namely the equality $\overline{K}^n=\overline{K}^{(n)}$ in the ring $A$ (as defined in Theorem \ref{intro2ndform}), for when $\mathcal{A}$ is equal to the ideal of the iterated Jacobian dual $(\mathcal{L}+I_d(B_n(\varphi)))$. This condition is always satisfied for ideals  with $\mu(I)=d+1$, i.e, for ideals of second analytic deviation one. \vskip 2mm
\begin{thm}
Let $R=k[x_1,\dots, x_d]$ and $I$ be a grade 2 perfect ideal. Suppose that the presentation matrix $\varphi$ of $I$ is almost linear, that is, all but the last column of $\varphi$ are linear and the last column consist of homogeneous entries of degree $n$. If $I$ satisfies $G_d$ and $\mu(I)=d+1$, then the defining ideal of $\mathcal{R}(I)$ satisfies
	$\mathcal{A}=\mathcal{L}+I_d(B_n(\varphi))=\mathcal{L}:(x_1,\dots, x_d)^n$.
\end{thm}
\vskip 1mm
We also examine certain algebraic properties such as the Cohen Macaulayness of the Rees algebra and invariants such as the relation type $\rt(I)$, of the ideal and Castelnuovo-Mumford regularity of the Rees algebra. 
\begin{thm}
Let $R=k[x_1,\dots, x_d]$ and $I$ be a grade 2 perfect ideal. Suppose that the presentation matrix $\varphi$ of $I$ is almost linear, that is, all but the last column of $\varphi$ are linear and the last column consist of homogeneous entries of degree $n$. If $I$ satisfies $G_d$ and $\mu(I)=d+1$, then 	$\rt (I)=\reg~\mathcal{F}(I)+1=\reg~\mathcal{R}(I)+1=n(d-1)+1$.
\end{thm}

This paper is organized as follows. In Section 2, we study prime ideals of the form $(\underline{a}\cdot \psi):(\underline{a})^\infty$ for a regular sequence $\underline{a}$ and a certain matrix $\psi$. In Section 3 we explain the setting of the problem. We also present two forms of the defining equations of the Rees algebra of grade 2 perfect ideals whose presentation matrix is almost linear. In Section 4, we explain the procedure of iterated Jacobian duals, and we give a condition for the equality of the defining ideal of $\mathcal{R}(I)$ and the ideal of an iterated Jacobian dual. In Section 5, we show that the defining ideal of the Rees algebra of grade 2 perfect ideals with second analytic deviation one, coincide with ideal of an iterated Jacobian dual. Algebraic properties such as Cohen-Macaulayness and the Castelnuovo-Mumford regularity of the Rees algebra of such ideals are also studied.
\subsection*{Acknowledgements}

We thank the makers of the computer algebra software \textit{Macaulay2}, which helped us verify our guesses and hopes.

We are immensely grateful to our advisor Prof. Bernd Ulrich for his insightful comments and observations, without which this work would not have come to fruition.  

\section{Prime Saturations}
The $G_d$ condition on a strongly Cohen Macaulay ideal $I \subset k[x_1, \ldots, x_d]$ implies that $\mathcal{A} = \mathcal{L} : (\underline{x})^\infty$ is a prime ideal of expected height. Under such conditions we would like to know generators of the colon ideal $\mathcal{L} : (\underline{x})$. If this colon ideal is a residual intersection, the work of Huneke and Ulrich shows that $\mathcal{L} : (\underline{x}) = \mathcal{L} + I_d(B(\varphi))$, where $B(\varphi)$ is a Jacobian Dual of $\varphi$ \cite{HU1}.

In the case where $I$ is a perfect ideal of grade two which is almost linearly presented, we show that $\mathcal{L} : (\underline{x})$ is a residual intersection by showing that $I_d(B(\varphi '))$ attains the maximum height. Recall that $\varphi '$ is the matrix consisting of the linear columns of $\varphi$. Notice that $\mathcal{L} : (\underline{x})^\infty = (\underline{x}\cdot B(\varphi)): (\underline{x})^\infty$. We first show that this saturation being a prime ideal of expected height implies that $(\underline{x}\cdot B(\varphi ' )): (\underline{x})^\infty$ is also a prime ideal of expected height. We then show that this along with the fact that the entries of $B(\varphi ')$ are in $k[T_1, \ldots, T_m]$ implies that $(\underline{x}\cdot B(\varphi ')):(\underline{x})^\infty = (\underline{x}\cdot B(\varphi ')):(\underline{x}) = (\underline{x}\cdot B(\varphi ')) + I_d(B(\varphi '))$ and that $I_d(B(\varphi '))$ attains the maximum height. We give these results about prime saturations in a more general setting than will be needed for later applications.

Our first result makes use of the following exercise from Kaplansky. \cite[Exercise 5a p.7]{KAP}

\begin{lem}\label{KapLemma} Let $T$ be a commutative ring and $Q=(q)$ a principal prime ideal. If $P$ is a prime ideal properly contained in $Q$, then $P \subset \bigcap_{i=1}^\infty Q^i$ \end{lem}

\begin{proof}
Let $p\in P$. We prove by induction on $i$ that $p \in Q^i$ for all $i$. The case $i=1$ is clear as $p \in P\subset Q = Q^1$. Now assume that $p \in Q^{i-1}=(q^{i-1})$ for $i>1$. Then $p=t q^{i-1}$ for some $t \in T$. $q^{i-1} \notin P$ since $P \neq Q$. Thus $t \in P$ since $t q^{i-1} \in P$ and $P$ is a prime ideal. So $t \in P \subset Q = (q)$ and therefore $p = t q^{i-1} \in (q^i) = Q^i$.
\end{proof}

\begin{prop}\label{godown} Let $S$ be a positively graded Noetherian ring, $\mathfrak{a}$ an $S_0$-ideal and \\$(L_1, \ldots, L_s) \subset \mathfrak{a} S \cap S_+$ an $S$-ideal with $(L_1, \ldots, L_{s-1})$ a homogeneous ideal. If $(L_1, \ldots, L_s) :\mathfrak{a}^\infty$ is a prime ideal of height $s$, then $(L_1, \ldots, L_{s-1}):\mathfrak{a}^\infty$ is a prime ideal of height $s-1$. \end{prop}
\begin{proof}

The ideal $(L_1, \ldots, L_s):\mathfrak{a}^\infty$ contains some minimal prime $P$ of $(L_1, \ldots, L_{s-1})$. We will show that this prime is precisely $(L_1, \ldots, L_{s-1}) : \mathfrak{a}^\infty$. Let $a \in \mathfrak{a}\setminus \{0\}$ and consider $S_a$. Note that $a \notin (L_1, \ldots, L_s) : \mathfrak{a}^\infty$ since otherwise $a\mathfrak{a}^v \subset (L_1, \ldots, L_s)$ for some $v \in \mathbb{N}$ which is impossible by degree considerations. Now $(L_1, \ldots, L_{s-1})_a \subset P_a \subset ((L_1, \ldots, L_s) : \mathfrak{a}^\infty)_a = (L_1, \ldots, L_s)_a $ and the second inclusion is proper since by Krull's Altitude Theorem, $P$ has height at most $s-1$. Consider the ring $T = S_a/(L_1, \ldots, L_{s-1})_a$. Applying Lemma \ref{KapLemma} to the ring $T$ with $P=P_aT$ and $Q = ((L_1, \ldots, L_s) : \mathfrak{a}^\infty)_aT = (L_1, \ldots, L_s)_aT = L_sT$, we see that $P_aT \subset \bigcap_{i=1}^\infty L_s^iT$. Notice that the ring $T = S_a/(L_1, \ldots, L_{s-1})_a$ is positively graded and that $L_sT \subset T_+$. Therefore $P_aT\subset \cap_{i=1}^\infty L_s^iT = 0$.
 So $P_a = (L_1, \ldots, L_{s-1})_a$. As this is true for every $a \in \mathfrak{a}$ and $\mathfrak{a}$ is finitely generated this shows that $\mathfrak{a}^k P \subset (L_1, \ldots, L_{s-1})$ for some $k \in \mathbb{N}$. Thus $P = (L_1, \ldots, L_{s-1}) : \mathfrak{a}^\infty$.\end{proof}

This lemma is useful on its own, but we will now use it to prove the main result of this section by induction.

\begin{thm}\label{goup} Let $R$ be a Cohen Macaulay local ring containing a field $k$ and $\underline{a} = a_1, \ldots, a_r$ an $R$-regular sequence with $r>0$. Let $S = R[T_1, \ldots, T_m]$ with $T_1, \ldots, T_m$ indeterminates over $R$ and $\psi$ be an $r \times s$ matrix with entries in $k[T_1, \ldots, T_m]$ so that each column consists of homogeneous elements of the same positive degree. If $(\underline{a} \cdot \psi) :_S (\underline{a})^\infty$ is a prime ideal of height $s$, then $I_r(\psi)$ is a prime ideal of $k[T_1, \ldots, T_m]$ of height $\mathrm{max}\{0, s-r+1\}$ and $(\underline{a}\cdot \psi) :_S (\underline{a})$ is a geometric residual intersection. Furthermore, \[(\underline{a}\cdot \psi) :_S (\underline{a})^\infty = (\underline{a}\cdot \psi) :_S (\underline{a}) = (\underline{a}\cdot \psi) + I_r(\psi).\]\end{thm}

\

\noindent We first reduce the proof of the theorem to a height computation:

\begin{obs}\label{geometric} Theorem \ref{goup} follows once it has been shown that \\$\mathrm{ht}(I_r(\psi)) = \mathrm{max}\{0, s-r+1\}$. \end{obs}

\begin{proof}
Assume that $\mathrm{ht}(I_r(\psi)) = \mathrm{max}\{0, s-r+1\}$.
First, note that $(\underline{a})^t \not\subset (\underline{a} \cdot \psi)$ for any $t \in \mathbb{N}$ by degree considerations. Thus $(\underline{a}) \not\subset (\underline{a} \cdot \psi) :_S (\underline{a})^\infty$. Then, since $(\underline{a})^t \cdot (\underline{a} \cdot \psi) :_S (\underline{a})^\infty \subset (\underline{a} \cdot \psi)$ for some $t \in \mathbb{N}$, we have that $(\underline{a} \cdot \psi) :_S (\underline{a})^\infty$ is the unique associated prime of $(\underline{a} \cdot \psi)$ not containing $(\underline{a})$. Since any minimal prime of $(\underline{a} \cdot \psi) :_S (\underline{a})$ contains either $(\underline{a})$ or $(\underline{a} \cdot \psi) :_S (\underline{a})^\infty$, any such prime contains either $(\underline{a}) + I_r(\psi)$ or $(\underline{a} \cdot \psi) :_S (\underline{a})^\infty$. But $(\underline{a} \cdot \psi) :_S (\underline{a})^\infty$ has height $s$, and $(\underline{a})+I_r(\psi)$ has height $s+1$ since $\underline{a}$ is a regular sequence on $S/I_r(\psi)$ (recall $I_r(\psi) \subset k[T_1, \ldots, T_m]$). Thus $\mathrm{ht}((\underline{a} \cdot \psi) :_S (\underline{a})) \geq s$ and $\mathrm{ht}((\underline{a}) + (\underline{a} \cdot \psi) :_S (\underline{a})) \geq s+1$, so $(\underline{a} \cdot \psi) :_S (\underline{a})$ is a geometric residual intersection. Then by \cite[4.8]{BKM}, \[(\underline{a} \cdot \psi) :_S (\underline{a}) = (\underline{a}\cdot \psi) + I_r(\psi).\] Also from \cite[1.5]{HU2} $(\underline{a} \cdot \psi) :_S (\underline{a})$ is Cohen Macaulay, hence an unmixed ideal of height $s$. Thus none of its minimal primes contain $(\underline{a})$ and therefore $(\underline{a}\cdot \psi) :_S (\underline{a})^\infty$ is its only associated prime. Then $(\underline{a} \cdot \psi) :_S (\underline{a})^\infty = (\underline{a} \cdot \psi) :_S (\underline{a})$ as both are equal to $(\underline{a} \cdot \psi)$ locally at $(\underline{a} \cdot \psi) :_S (\underline{a})^\infty$.

Lastly, since $(\underline{a} \cdot \psi) + I_r(\psi)$ is a prime ideal, $((\underline{a} \cdot \psi) + I_r(\psi))\cap k[T_1,\ldots,T_m]=I_r(\psi)$ is a prime ideal of $k[T_1, \ldots, T_m]$ as well.
\end{proof}

\begin{proof}[Proof of Theorem \ref{goup}]

We show that $\mathrm{ht}(I_r(\psi))=\mathrm{max}\{0,s-r+1\}$ by induction on $s$. Since the extension $k[T_1, \ldots, T_m] \subset S$ is flat, we may compute the height in the ring $k[T_1, \ldots, T_m]$ or in $S$.

If $s=0$, then certainly $\mathrm{ht}~I_r(\psi) = \mathrm{ht \ }(0) = 0$. Now $(\underline{a})^k(0 :_S (\underline{a}^\infty)) = 0$ for some $k \in \mathbb{N}$. As $r>0$, $(\underline{a})$ is not contained in any associated prime of the ring. This shows that $0:_S(\underline{a}) \subset 0:_S(\underline{a})^\infty = 0 = (\underline{a} \cdot \psi).$

Now assume that the result holds for matrices with less than $s$ columns and that $\psi$ is an $r$ by $s$ matrix as in the theorem so that $(\underline{a}\cdot \psi) :_S (\underline{a})^\infty$ is prime ideal of height $s>0$.


Arrange the columns of $\psi$ so that the degrees of the columns descend from left to right and let $\psi = [ \psi'  \ | \  \psi_{s} ]$ with $\psi'$ an $r \times s-1$ matrix and $\psi_s$ an $r \times 1$ matrix.

By Proposition \ref{godown}, $(\underline{a} \cdot \psi' ) :_S (\underline{a})^\infty$ is prime ideal of height $s-1$. Then by the induction hypothesis $\hght I_r(\psi')=\max\{0,s-r\}$, which by Observation \ref{geometric} shows that $(\underline{a} \cdot \psi' ) :_S (\underline{a})$ is a geometric residual intersection and 
 \[(\underline{a} \cdot \psi') :_S (\underline{a})^\infty = (\underline{a} \cdot \psi') :_S (\underline{a}) = (\underline{a} \cdot \psi') + I_r(\psi').\]

We prove the induction step in cases. Since $I_r(\psi ')$ and $I_r(\psi)$ are involved, we consider
the cases $s<r$, $s=r$, and $s>r$.

If $s<r$, $I_r(\psi)=0$ hence $\mathrm{ht}(I_r(\psi))=0$.
 
If $s =r$, we have $s-1 < r$. Thus $(\underline{a} \cdot \psi') = (\underline{a} \cdot \psi') + I_r(\psi ')$ is prime and $\underline{a} \cdot \psi_s \notin (\underline{a} \cdot \psi')$, as otherwise $(\underline{a} \cdot \psi ) :_S (\underline{a})^\infty$ is a minimal prime of $(\underline{a} \cdot \psi ')$ and is therefore a prime ideal of height at most $r-1$. So $(\underline{a} \cdot \psi)$ is a complete intersection of height $r$ and thus $(\underline{a} \cdot \psi) + I_r(\psi)$ has height at least $r$. Then $(\underline{a} \cdot \psi) :_S (\underline{a}) = (\underline{a} \cdot \psi) + I_r(\psi)$ \cite[3.3]{N1}. 
Note now that $\mathrm{det}\psi \neq 0$ since $\underline{a}$ is contained in a minimal prime of $(\underline{a}\cdot \psi)$. Hence $\mathrm{ht}I_r(\psi)=1=s-r+1$.
 
 If $s> r$, we have that $\mathrm{ht}~I_r(\psi') = s-r$. Since $(\underline{a}\cdot \psi '):_S (\underline{a})$ is a geometric residual intersection, $(\underline{a}) \cap ((\underline{a}\cdot \psi ') + I_r(\psi ')) = (\underline{a}\cdot \psi ')$ by \cite[1.5]{HU2}.
 
We show that $I_r(\psi)$ contains a regular element modulo $I_r(\psi ')$. If $f I_r(\psi) \subset I_r(\psi ')$ for some some $f \in k[T_1, \ldots, T_m]$, we will show that $f \in I_r(\psi ')$. As $f I_r(\psi) \subset I_r(\psi ')$,  $f I_r(\psi) \subset (\underline{a} \cdot \psi') + I_r(\psi ')$ which is prime. Therefore $f \in (\underline{a} \cdot \psi') + I_r(\psi ')$ or $I_r(\psi) \subset (\underline{a} \cdot \psi') + I_r(\psi ')$ which implies that $f \in I_r(\psi ')$ or $I_r(\psi) \subset I_r(\psi ')$ since these elements are in $k[T_1, \ldots, T_m]$.
 
 Assume by way of contradiction that $I_r(\psi) \subset I_r(\psi ')$. Consider the $r\times r$ submatrix consisting of the last $r$ columns of $\psi'$. Repeatedly applying Proposition \ref{godown} and making use of the case $s=r$, we see that the determinant of this submatrix is nonzero. So one of the $r-1 \times r-1$ minors from columns $s-r+1, \ldots , s-1$ is nonzero. Call this minor $\delta$ and assume it comes from deleting row $i$. Let $\Delta$ be the $r \times r$ minor of $\psi$ involving columns $s-r+1, \ldots, s$. Then $a_i \Delta \subset (\underline{a} \cdot \psi')$ as $\Delta \in I_r(\psi) \subset I_r(\psi ')$. But modulo $(\underline{a} \cdot \psi ')$, $a_i  \Delta \equiv \delta (\underline{a}\cdot \psi_s)$ by Cramer's rule. So $\delta (\underline{a} \cdot \psi_s) \subset (\underline{a}\cdot \psi ') + I_r(\psi ')$. But $\underline{a} \cdot \psi_s \notin (\underline{a}\cdot \psi ') + I_r(\psi ')$, otherwise $\underline{a} \cdot \psi_s \in (\underline{a}) \cap ((\underline{a}\cdot \psi ') + I_r(\psi ')) = (\underline{a}\cdot \psi ')$, contradicting the fact that $\mathrm{ht}((\underline{a} \cdot \psi) :_S (\underline{a})^\infty) = s$. Also $\delta \notin (\underline{a}\cdot \psi ') + I_r(\psi ')$, otherwise $\delta \in I_r(\psi ')$, which is impossible by degree considerations. This is a contradiction to the fact that $(\underline{a}\cdot \psi ') + I_r(\psi ')$ is a prime ideal. Thus $\mathrm{ht}~I_r(\psi) = s-r+1$.
 \end{proof}

\section{Almost Linearly Presented Ideals}
In this section, we prove our two main descriptions of the defining ideal of $\mathcal{R}(I)$ for almost linearly presented ideals $I$.
Our first description is as the saturation of the defining ideal of $\mathrm{Sym}(I)$ with respect to the ideal of variables of the polynomial ring. We show that the saturation stabilizes at the degree of the last column of the presentation matrix.

\begin{set}\label{origset} Let $R=k[x_1, \ldots, x_d]$ be a standard graded polynomial ring over a field $k$, and let $I=(\alpha_1, \ldots, \alpha_m) \subset R$ be a height $2$ perfect ideal with almost linear presentation matrix. That is, with respect to $\alpha_1, \ldots, \alpha_m$, $I$ has a Hilbert-Burch matrix of the form $\varphi = \left[v_1 | \cdots  | v_{m-1} \right]$, where $v_i \in (R_1)^m$ for $1\leq i \leq m-2$, and $v_{m-1} \in (R_n)^m$.

Assume that $m>d$ and $I$ also satisfies the $G_d$ condition, i.e. $\mu(I_P) \leq $ ht$P$ for all primes $P \in V(I)$ with ht$P < d$.\end{set}

\begin{notation}
Let $\varphi ' = \left[ v_1 | \cdots | v_{m-2}\right]$, and $B(\varphi), B(\varphi ')$ be a Jacobian dual of $\varphi, \varphi '$ respectively. That is $$[x_1\cdots x_d]\cdot B(\varphi)=[T_1\cdots T_m]\cdot \varphi\mbox{ and }[x_1\cdots x_d]\cdot B(\varphi')=[T_1\cdots T_m]\cdot \varphi'.$$
Let $S = R[T_1, \ldots, T_m]$ and $\mathcal{A}$ be the defining ideal of $\mathcal{R}(I)$, that is $\mathcal{A} =$ ker $\psi$ where $\psi : S \longrightarrow \mathcal{R}(I)=R[It]$ is the $R$-algebra homomorphism with $\psi(T_i)=\alpha_i t$.

Let the $S$-ideal $\mathcal{L}$ be the defining ideal for $\mathrm{Sym}(I)$, namely $\mathcal{L} = (L_1, \ldots, L_{m-2}, g)$ where $\left[T_1 \  \cdots \  T_m\right]\cdot \varphi = \left[L_1 \ \cdots \  L_{m-2} \  g \right]$. 
\end{notation}

\begin{rmkk} We assume in Setting \ref{origset} that $m>d$. For, if $m \leq d$ then $I$ satisfies the $G_\infty$ condition ($\mu(I_P) \leq $ ht $P$ for all $P \in \mathrm{Spec}(R)$) and hence by \cite[2.6]{HSV1} $I$ is of linear type, i.e. $\mathcal{A} = \mathcal{L}$. \end{rmkk}

The $G_d$ condition alone is enough to guarantee one description of the defining ideal.

\begin{obs}\label{0thform} With the assumptions of Setting \ref{origset}, $\mathcal{A} = \mathcal{L} : (x_1, \ldots, x_d)^\infty$.
\end{obs}
\begin{proof}
Let $s \in (x_1, \ldots, x_d) \subset R$. Then because of the $G_d$ condition, $I_s$ satisfies $G_\infty$ in $R_s$. So $I_s$ is of linear type by \cite[2.6]{HSV1}. This means that $s^t\mathcal{A} \subset \mathcal{L}$ for some $t \in \mathbb{N}$. But clearly $\mathcal{L} : (x_1, \ldots, x_d)^i \subset \mathcal{A}$ for all $i$ since $(x_1, \ldots, x_d)^i \not\subset \mathcal{A}$ and $\mathcal{A}$ is a prime ideal. So $\mathcal{A} = \mathcal{L} : (x_1, \ldots, x_d)^\infty$.
\end{proof}
 
We will show in Theorem \ref{1stform} that in fact $\mathcal{A} = \mathcal{L} :_S (x_1, \ldots, x_d)^n$ where $n$ is the degree of the entries in the last column of the presentation matrix $\varphi$.

 \begin{obs}\label{dimsym} With the assumptions of Setting \ref{origset}, $\mathrm{ht}\mathcal{L}=d$.
 \end{obs}
 \begin{proof} Recall that $\mathrm{Sym}(I) \cong S/\mathcal{L}$, where $S = R[T_1, \ldots, T_m]$. By \cite[2.6]{HR}, \[\mbox{dim}~\mathrm{Sym}(I) =  \mbox{sup}  \left\{ \mu(I_P) + \mathrm{dim}(R/P)  \ | \ P \in \mathrm{Spec}(R) \right\} .\]
 
 To find the supremum on the right hand side, we compute $\mu(I_P)+\mathrm{dim}(R/P)$ in different cases. For $P \notin V(I)$, $\mu(I_P) =1$ so this number is less than or equal to $d+1$. If $\mathrm{ht}P<d$ and $P \in V(I)$, $\mu(I_P) + \mbox{dim}(R/P) \leq \mbox{ht}P + \mbox{dim}(R/P) = d$. When $\mbox{ht}P=d$, this number is the number of generators of $I$ which is $m$. Thus dim $\mathrm{Sym}(I) = m$ and ht $\mathcal{L} = d$.
\end{proof}

\begin{obs}\label{1stcolon} With the assumptions of Setting \ref{origset},  $\mathcal{L} : (x_1, \ldots, x_d) = \mathcal{L}+I_d(B(\varphi))$.
\end{obs}
\begin{proof} We know that $\mathrm{dim}~ \mathcal{R} (I) = d+1$, so ht $\mathcal{A} = m-1$.

As $\mathcal{A} = \mathcal{L} : (x_1, \ldots, x_d)^\infty = (\underline{x}\cdot B(\phi)) : (\underline{x})^\infty$ is a prime ideal of height $m-1$, by Proposition \ref{godown}, $(\underline{x}\cdot B(\varphi ')) : (\underline{x})^\infty$ is a prime ideal of height $m-2$. Then since the entries of $B(\varphi ')$ are in $k[T_1, \ldots, T_m]$, ht $I_d(B(\varphi')) = m-d-1$ by Theorem \ref{goup}. So \[\mathrm{ht}\left( (x_1, \ldots, x_d)+I_d(B(\varphi '))\right) = m-d-1+d = m-1\] since $x_1, \ldots, x_d$ is a regular sequence modulo $I_d(B(\varphi '))$.

As $(x_1, \ldots, x_d)^t \mathcal{A} \subset \mathcal{L} \subset \mathcal{L} + I_d(B(\varphi))$ for some $t$, every minimal prime of $\mathcal{L} + I_d(B(\varphi))$ contains either $(x_1, \ldots, x_d)+I_d(B(\varphi'))$ or $\mathcal{A}$.  So ht$\left(\mathcal{L} + I_d(B(\varphi))\right) \geq m-1$. Since $\mathcal{L} + I_d(B(\varphi)) \subset \mathcal{L} : (x_1, \ldots, x_d)$ it follows that $\mathrm{ht} (\mathcal{L} : (x_1, \ldots, x_d)) \geq m-1$. Thus by \cite[1.5 and 1.8]{HU3}, $\mathcal{L} : (x_1, \ldots, x_d) = \mathcal{L} + I_d(B(\varphi))$.
\end{proof}

\noindent Now we make use of a Lemma. The generalized version of the lemma presented here was formulated by Professor Bernd Ulrich.

\begin{lem} Let $R$ be a Noetherian ring and $I \subset R$ an $R$-ideal. If $I^n \cap (0 : I) = 0$ for some $n \in \mathbb{N}$, then $I^n (0 : I^\infty ) = 0$. \end{lem}
\begin{proof}
Let $w \in \mathbb{N}$ be smallest so that $I^w(0 : I^\infty) = 0$. Assume by way of contradiction that $w>n$. Then $0 \neq I^{w-1}(0 : I^\infty)$. But $I^{w-1}(0 : I^\infty) \subset I^n \cap (0:I) = 0$ \quad (Since $w-1 \geq n$, and $I^w(0:I^\infty)=0$).
\end{proof}

\begin{thm}\label{1stform} Use the assumptions of Setting \ref{origset}, and in particular let $n$ be the degree of the entries in the last column of $\varphi$. One has $\mathcal{A} = \mathcal{L} : (x_1, \ldots, x_d)^n$. \end{thm}

\begin{proof}
From Observation \ref{0thform}, $\mathcal{A} = \mathcal{L} : (x_1, \ldots, x_d)^\infty$. Now $\mathcal{L} : ((x_1, \ldots, x_d)+\mathcal{L}) = \mathcal{L} : (x_1, \ldots, x_d) = \mathcal{L} + I_d(B(\varphi))$ by Observation \ref{1stcolon}.

Also, $((x_1, \ldots, x_d)^n + \mathcal{L}) \cap \left(\mathcal{L} + I_d(B(\varphi))\right) = \mathcal{L}$. This is because all ideals involved are bi-homogeneous and $I_d(B(\varphi))$ is generated by elements of $(\underline{x})$-degree $\leq n-1$. Thus any element of the intersection with $(\underline{x})$-degree $\geq n$ is in $\mathcal{L} + (x_1, \ldots, x_d)^n \cap I_d(B(\varphi)) \subset \mathcal{L} + (x_1, \ldots, x_d)I_d(B(\varphi)) \subset \mathcal{L}$. Also any element of the intersection with $(\underline{x})$-degree $< n$ is in $\mathcal{L}$.

Now applying the previous lemma to the image of $(x_1, \ldots, x_d)$ in the ring $\mathrm{Sym}(I)$ we see that \[((x_1, \ldots, x_d)^n + \mathcal{L})\mathcal{A} = ((x_1, \ldots, x_d)^n + \mathcal{L})(\mathcal{L} : (x_1, \ldots, x_d)^\infty) \subset \mathcal{L}.\] Thus $\mathcal{L} : (x_1, \ldots, x_d)^n \subset \mathcal{A} \subset \mathcal{L} : (x_1, \ldots, x_d)^n$.
\end{proof}

 \ 
 
 Next we prove a different description of the defining ideal. Following the path laid in \cite{KPU1} we find a ring that surjects onto the Rees algebra so that the kernel is a height $1$ prime ideal in that ring. We will use this description in the following sections to compute many elements of the defining ideal of the Rees ring and, in a special case, the entire ideal.
 
In the remainder of this section we will use the assumptions of Setting \ref{origset} along with the following notation:
 \begin{notation}\label{2ndnot} Let $J$ be the $S$-ideal $(L_1, \ldots, L_{m-2})+I_d(B(\varphi '))$ and let $A$ be the ring $S/J$.  Let \ $\bar{ \ }$ denote images in the ring $A$. Let $B$ be the $d-1 \times m-2$ matrix obtained by deleting the last row from $B(\varphi ')$. Define the $S$-ideal $K$ to be $(L_1, \ldots, L_{m-2}) + I_{d-1}(B) + (x_d)$.
 
 \end{notation}
 We now give a description of $\overline{\mathcal{A}}$ as an $A$-ideal.
 
  \ 
 \begin{obs}\label{KisCM} The ring $A$ is a Cohen Macaulay domain of dimension $d+2$, and the ideals $\overline{K}$ and $\overline{(x_1, \ldots, x_d)}$ are Cohen Macaulay $A$-ideals of height $1$. Furthermore, $\overline{(x_1, \ldots, x_d)}$ is a prime ideal.\end{obs}
 \begin{proof}
  As before, $(\underline{x}\cdot B(\varphi ')) :_S (\underline{x})^\infty$ is a prime ideal of height $m-2$ by Proposition \ref{godown}. Since the entries of $B(\varphi ')$ are in $k[T_1, \ldots, T_m]$, \[ (\underline{x}\cdot B(\varphi ')) :_S (\underline{x})^\infty = (\underline{x}\cdot B(\varphi ')) :_S (\underline{x}) = (\underline{x}\cdot B(\varphi ')) + I_d(B(\varphi ')) = J\] is a prime ideal of height $m-2$ that is a residual intersection, hence Cohen-Macaulay by Theorem \ref{goup}. Thus $A$ is a Cohen-Macaulay domain of dimension $d+2$.
   
   Also $K$ has height at least $m-1$ as $J \subset K$ and $x_d \in K \setminus J$. Notice that $K$ can also be written as $(\tilde{L}_1, \ldots, \tilde{L}_{m-2}) + I_{d-1}(B) + (x_d)$ where $\big[\tilde{L}_1 \ \cdots \ \tilde{L}_{m-2}\big] = \left[x_1 \ \cdots \ x_{d-1}\right]\cdot B$. Therefore $(\tilde{L}_1, \ldots, \tilde{L}_{m-2}) + I_{d-1}(B)$ has height at least $m-2$. This means that it has height equal to $m-2$, is Cohen-Macaulay, and $(\tilde{L}_1, \ldots, \tilde{L}_{m-2}) + I_{d-1}(B) = (\tilde{L}_1, \ldots, \tilde{L}_{m-2}) :_S (x_1, \ldots, x_{d-1})$ (\cite[1.5 and 1.8]{HU3}). Then $K$ has height exactly $m-1$ and is also Cohen-Macaulay. This means that in the Cohen Macaulay ring $A$, $\overline{K}$ is a height $1$ ideal that is Cohen Macaulay.
    
    By Proposition \ref{godown} and Theorem \ref{goup}, ht$(I_d(B(\varphi '))) = m-d-1$ and thus $I_d(B(\varphi '))$ is Cohen-Macaulay \cite[5.2]{Ea}. But $x_1, \ldots, x_d$ is a regular sequence on $I_d(B(\varphi '))$, so $(x_1, \ldots, x_d) + I_d(B(\varphi ')) = (x_1, \ldots, x_d) + J$ is Cohen-Macaulay of height $m-1$. Since $J$ is a prime ideal that is homogeneous with respect to $(x_1, \ldots, x_d)$, it follows that $(x_1, \ldots, x_d)+J$ is a prime ideal also.
\end{proof}

  
\begin{lem}\label{colon} \[\overline{(x_1, \ldots, x_d)}^i = \overline{(x_1, \ldots, x_d)}^{(i)} = (\overline{x_d}^i) :_A \overline{K}^{(i)} \mbox{ and } \] \[\overline{K}^{(i)} = (\overline{x_d}^i):_A \overline{(x_1, \ldots, x_d)}^{(i)} \mbox{ for all }i \in \mathbb{N}.\]
\end{lem}
\begin{proof}
Temporarily setting the $x_i$'s degrees to $1$ and the $T_i$'s degrees to $0$, we see that $\mathrm{gr}_{(\overline{\underline{x}})}(A) \cong A$, a domain. Thus $(\overline{\underline{x}})^i = (\overline{\underline{x}})^{(i)}$.


 We now show the second equality.
 
In the ring $A$, $\overline{(x_1, \ldots, x_d)}\overline{K} \subset (\overline{x_d})$. So $\overline{(x_1, \ldots, x_d)}^i\overline{K}^i \subset (\overline{x_d}^i)$. Then localizing at height $1$ primes, we see that $\overline{(x_1, \ldots, x_d)}^{(i)}\overline{K}^{(i)}\subset (\overline{x_d}^i)$. Hence $\overline{(x_1, \ldots, x_d)}^{(i)} \subset (\overline{x_d}^i) :_A \overline{K}^{(i)}$. Expressing $K = (\tilde{L}_1, \ldots, \tilde{L}_{m-2}) + I_{d-1}(B) + (x_d)$ as before and recalling that $(\tilde{L}_1, \ldots, \tilde{L}_{m-2}) + I_{d-1}(B) = (\tilde{L}_1, \ldots, \tilde{L}_{m-2}):_S (x_1, \ldots, x_{d-1})$ has height $m-2$, we must have $0 \neq I_{d-1}(B)\subset k[\underline{T}]$. For, if $I_{d-1}(B)=0$ then $(x_1,\ldots,x_{d-1})$ contains a regular element on $S/(\tilde{L}_1, \ldots, \tilde{L}_{m-2})$ and hence $d-1=\hght (x_1,\ldots,x_{d-1})>\hght (\tilde{L}_1, \ldots, \tilde{L}_{m-2})=m-2$, contradicting $m>d$. Now by degree considerations, $\overline{K} \not\subset \overline{(x_1, \ldots, x_d)}$. Since $\overline{(x_1, \ldots, x_d)}$ is the unique associated prime of $\overline{(x_1, \ldots, x_d)}^{(i)}$, it follows that $(\overline{x_d}^i):_A \overline{K}^{(i)} \subset \overline{(x_1, \ldots, x_d)}^{(i)}$.

 Similarly, $\overline{K}^{(i)} = ({\overline{x_d}}^i) :_A \overline{(x_1, \ldots, x_d)}^{(i)}$.
 \end{proof}
 
 Now define the $A$-ideal $\mathcal{D}= \frac{\bar{g}\overline{K}^{(n)}}{\overline{x_d}^n}$. This is an $A$-ideal by Lemma \ref{colon}, since $g \in (x_1, \ldots, x_d)^n$. Moreover, $\mathcal{D} \subset \overline{\mathcal{A}}$ as $\bar{g} \in \mathcal{A}$, $\overline{x_d} \notin \overline{\mathcal{A}}$, and $\overline{\mathcal{A}}$ is prime.
 
\begin{thm}\label{2ndform} With the assumptions of Setting \ref{origset}, the $A$-ideals $\mathcal{D}$ and $\overline{\mathcal{A}}$ are equal.\end{thm}  
  
\begin{proof}
Since $A$ is Cohen-Macaulay, a proper $A$-ideal $\mathfrak{b}$ satisfies the \textit{Serre} condition $S_2$ as an $A$-module if and only if it is unmixed of height one. This follows from the depth lemma applied to the exact sequence $0 \longrightarrow \mathfrak{b} \longrightarrow A \longrightarrow A/\mathfrak{b} \longrightarrow 0$. As the condition $S_2$ is preserved under isomorphism, this shows that for proper $A$-ideals, the property of being height $1$ unmixed is preserved under isomorphism. Thus $\mathcal{D}$ is height one unmixed since it is a proper ideal that is isomorphic to $K^{(n)}$.


As $\mathcal{D} \subset \overline{\mathcal{A}}$, to prove equality it is enough to prove that they are equal locally at associated primes of $\mathcal{D}$, which are of height $1$.

Notice that as $(\overline{\underline{x}})_{(\overline{\underline{x}})} = (\bar{x}_d)_{(\overline{\underline{x}})}$, the only $(\overline{\underline{x}})$-primary ideals of $A$ are the symbolic powers of $(\overline{\underline{x}})$, which are the powers. Then $(g)_{(\overline{\underline{x}})} = (\overline{\underline{x}})_{(\overline{\underline{x}})}^i$ for some $i$.
But then $g \in (\overline{\underline{x}})^i$. From this we see that $i\leq n$.

Now locally at a height $1$ prime $p$ not equal to $(\overline{\underline{x}})$, $\overline{K}_p = (\overline{x_d})_p$ and $\mathcal{D}_p = (\bar{g})_p = \overline{\mathcal{A}}_p$. For $p = (\overline{\underline{x}})$, $\overline{\mathcal{A}}_p = A_p$ and we just need to check that $\mathcal{D}_p = A_p$ as well. But $(\bar{g})_p = (\overline{\underline{x}})_p^i = (\overline{x_d}^i)_p$ for some $i\leq n$ and $\overline{K}_p = A_p$. So $\mathcal{D}_p = (\overline{x_d}^i)A_p/\overline{x_d}^n \supset A_p$.
\end{proof}

Now, using the ring $A$ we show that $n$ is the smallest possible integer for a description as in Theorem \ref{1stform}.

\begin{rmk}\label{nisMinimal} With the assumptions of Setting \ref{origset}, $n$ is the smallest integer so that $\mathcal{A} = \mathcal{L} : (x_1, \ldots, x_d)^n$.\end{rmk}
\begin{proof}
Assume that there is an $i \in \mathbb{N}$ with $i<n$ so that $\mathcal{A} = \mathcal{L} : (x_1, \ldots, x_d)^i$. Then in the ring $A$, $\overline{(x_1, \ldots, x_d)}^i \overline{\mathcal{A}} \subset (\bar{g})$. Localizing at the prime $\overline{(x_1, \ldots, x_d)}$ we obtain ${\overline{(x_1,\ldots,x_d)}^i}_{(\underline{x})} \subset (\bar{g})_{(\underline{x})}$. As $g \in (x_1, \ldots, x_d)^i$, this shows that ${\overline{(x_1,\ldots,x_d)}^i}_{(\underline{x})} = (\bar{g})_{(\underline{x})}$. Similarly ${\overline{(x_1,\ldots,x_d)}^n}_{(\underline{x})} \subset (\bar{g})_{(\underline{x})}$
Thus \[\overline{(x_1, \ldots, x_d)}^i = \overline{(x_1, \ldots, x_d)}^{(i)} = \overline{(x_1, \ldots, x_d)}^{(n)} =  \overline{(x_1, \ldots, x_d)}^n .\]
This is a contradiction.

\end{proof}

\section{Iterated Jacobian dual}
In this section we present an algorithm to extend the Jacobian dual matrix and introduce the notion of Iterated Jacobian dual. The minors of these matrices, help us construct more generators for the defining ideal of the Rees algebra, especially for those which are not of the expected form.

First we define the iterated Jacobian dual of an arbitrary matrix $\phi$, in a Noetherian ring $R$. We then apply the setting of \ref{origset} and present a condition for the equality of the ideal of iterated Jacobian dual and the defining ideal of $\mathcal{R}(I)$.

\

\noindent\textbf{Constructing the Iterated Jacobian dual}:
\\Let $R$ be a Noetherian ring. Consider a  presentation
\[R^s\xrightarrow{\phi}R^m.\]

Assume $I_1(\phi)\subseteq (a_1,\dots, a_r)$. Then there exists an $r\times s$ matrix $B(\phi)$, called a \textit{Jacobian dual} of $\phi$, with linear entries in $R[T_1, \dots, T_m]$ such that the following condition is satisfied
\begin{equation}\label{JacobianDual}
[T_1\cdots T_m]\cdot \phi=[a_1 \cdots a_r]\cdot B(\phi) .
\end{equation}
Though the existence of $B(\phi)$ is clear, it may not be uniquely determined. When $R$ is a polynomial ring, the uniqueness of $B(\phi)$ depends on the linearity of the presentation matrix $\phi$.

Let $\mathcal{L}$ denote the ideal defining the symmetric algebra $\sym(\coker\phi)$.
\begin{definition} \label{def iterated jacobian daul} Set $B_1(\phi)=B(\phi)$ and $\mathcal{L}_1=\mathcal{L}$. Suppose $(B_1(\phi),\mathcal{L}_1), \dots,  \\(B_{i-1}(\phi),\mathcal{L}_{i-1})$ have been inductively constructed such that, for $1\leq j\leq i-1$,  $B_j(\phi)$ are matrices with $r$ rows having homogeneous entries of constant degree along each column in $R[T_1,\dots,T_m]$ and $\mathcal{L}_{j}=(\underline{a}\cdot B_{j}(\phi)),~1\leq j\leq i-1$. 
	
	
	We now construct $(B_{i}(\phi),\mathcal{L}_{i})$. Let
	\begin{equation*}
	\mathcal{L}_{i-1}+(I_r(B_{i-1}(\phi))\cap (\underline{a}))=\mathcal{L}_{i-1}+(u_{1},\dots ,u_{l}) 
	\end{equation*}
	where $u_1,\dots,u_l$ are homogeneous in $R[T_1,\dots,T_m]$. Then there exists a matrix, C, having homogeneous entries of constant degree along each column in $R[T_1,\dots,T_m]$ such that 
	\[[u_{1}\cdots u_{l}]=[a_1\cdots a_r]\cdot C.\]
	Define $B_{i}(\phi)$, an \textit{i-th iterated Jacobian dual} of $\phi$, to be 
	\begin{equation}
 B_{i}(\phi)=[B_{i-1}(\phi)~|~C]
	\end{equation}
	where $|$ represents matrix concatenation. Now set $\mathcal{L}_{i}=(\underline{a}\cdot B_{i}(\phi))$.
\end{definition}
Notice that, by construction, $B_{i-1}(\phi)$ is a submatrix of $B_i(\phi)$ and $\mathcal{L}_{i-1}\subseteq\mathcal{L}_i$.
Supplementing the earlier observation, $B_{i}(\phi)$ may not be uniquely determined. Further, notice that the generating set $(u_{1},\dots ,u_{l})$ need not be unique, leading to different candidates for $B_{i}(\phi)$ of different sizes. Suppose 
\begin{equation}\label{first step of construction}
\mathcal{L}_{i-1}+(I_r(B_{i-1}(\phi))\cap (\underline{a}))=\mathcal{L}_{i-1}+(u_{1},\dots ,u_{l})=\mathcal{L}_{i-1}+(v_{1},\cdots ,v_{t})\end{equation} and suppose $B, B'$ satisfy
\begin{align}
\label{UniquenessofJacobianDual}
\begin{split}
[u_1 \cdots u_l]=[\underline{a}]\cdot C\mbox{~ and~} B=[B_{i-1}(\phi)~|~C]\\
[v_1 \cdots v_t]=[\underline{a}]\cdot C'\mbox{~ and~} B'=[B_{i-1}(\phi)~|~C']\\
\end{split}
\end{align}
For our purposes, we show, in Theorem \ref{Theorem of Uniqueness of Iterated Jacobian Dual}, that $\mathcal{L}+I_r(B)=\mathcal{L}+I_r(B')$ when $\underline{a}$ is a $R$-regular sequence. This will show that the ideal of the iterated Jacobian dual, $\mathcal{L}+I_r(B_{i+1}(\phi))$, depends only on the presentation matrix $\phi$ and the regular sequence $a_1,\dots,a_r$. It should be noted that $r$ should not be ``too big'', otherwise the matrix $B(\phi)$ (and hence $B_i(\phi)$) may have a row of zeros, which would trivialize the construction. 
\begin{rmk}\label{remarksIJD}
	\begin{enumerate}
		\item $\mathcal{L}_1=\mathcal{L}$ is a well defined $R[T_1,\dots,T_m]$-ideal because it is the ideal defining the symmetric algebra $\sym(\coker\phi)$. Assume that $\mathcal{L}_j,~1\leq j\leq i-1$ are well defined ideals. The candidates for $B_{i}(\phi)$, namely $B$ and $B'$, are constructed with the generators, $(u_1,\dots,u_l)$ and $(v_1,\dots,v_t)$ respectively. Now \eqref{first step of construction} guarantees that 
		\begin{equation*}                                  
		(\underline{a}\cdot B)=(\underline{a}\cdot B')
		\end{equation*}
		showing that $\mathcal{L}_{i}$ is a well-defined $R[T_1,\cdots,T_m]$-ideal.
		\item By definition, since $\mathcal{L}\subseteq\mathcal{L}_i$, it is clear that $\mathcal{L}+I_r(B_i(\phi))\subseteq \mathcal{L}_i+I_r(B_i(\phi))$. But notice that 
		\begin{align*}
		\begin{split}
		\mathcal{L}_i\subseteq &(\underline{a}\cdot B_{i-1}(\phi))+(\underline{a}\cdot C)\\
		=&\mathcal{L}_{i-1}+(u_1,\dots,u_l)\\
		\subseteq & \mathcal{L}_{i-1}+ I_r(B_{i-1}(\phi))\\
		\subseteq &\mathcal{L}_{i-1}+ I_r(B_i(\phi))\\
		\end{split}\end{align*}
		Successively, we can show that $\mathcal{L}_i\subseteq \mathcal{L}_{i-1}+I_r(B_i(\phi))\subseteq \mathcal{L}_{i-2}+I_r(B_i(\phi))\subseteq\cdots\subseteq \mathcal{L}+I_r(B_i(\phi))$. Thus $\mathcal{L}+I_r(B_i(\phi))= \mathcal{L}_i+I_r(B_i(\phi))$
		
	\end{enumerate}
\end{rmk}
Making use of \textit{Cramer's rule}, we  first prove a lemma which will be used liberally throughout this section.
\begin{lem}\label{LemmaofMinors}
	Let $R$ be a commutative ring. Let $[a_1 \cdots a_r]$ be a  $1\times r$ matrix and 
	$M$ be a $r\times r-1$ matrix with entries in $R$. Now let $M_t,~1\leq t\leq r$, be the $r-1\times r-1$ submatrix of $M$  obtained by removing the $t$-th row of $M$. Set $m_t=\mathrm{det~}M_t$. Then, in the ring $R/(\underline{a}\cdot M)$
	\begin{equation} \label{eqLemmaofMinors}
	\overline{a_t}\cdot \overline{m_k}=(-1)^{t-k}\overline{a_k}\cdot \overline{m_t},~ 1\leq t\leq r,~1\leq k\leq r
	\end{equation}
\end{lem}
\begin{proof}
	Let $M=(b_{ij})$. We start by writing it in the following form
	\begin{equation*}
	[a_1\cdots \hat{a_k}\cdots a_r]\cdot M_k=[g_{1}\cdots g_{r-1}]
	\end{equation*}
	Using \textit{Cramer's Rule}, we see that $a_t\cdot m_k=\mathrm{det~} M_{k_t},~ t\in \{1,\cdots ,\hat{k},\cdots, r\}$ where $M_{k_t}$ is a matrix got from $M_k$, by replacing the $t$-th row by $[g_{1}\cdots g_{r-1}]$. But in the ring $R/(\underline{a}\cdot M)$, 
	\begin{equation*}
	\overline{g_{i}}=-\overline{a_k}\overline{b_{ki}} 
	\end{equation*}
	Thus, in this ring, we have $\overline{a_t}\cdot\overline{m_k}=\overline{\mathrm{det~} M_{k_t}}=-\overline{a_k}\cdot \overline{m''}$, where $m''$ is the determinant of the matrix $M''$ whose rows are equal to that of $M_k$, except for the $t$-th row which is replaced by $[b_{k1}\cdots b_{kr}]$. Also, after $t-k-1$ row transposition of the $t$-th row of $M''$, we get $m''=(-1)^{t-k-1}m_t$, where $m_t$ is as described in the statement of the lemma. Putting all these observations together, we get $\overline{a_t}\cdot \overline{m_k}=-\overline{a_k}\cdot \overline{m''}=-\overline{a_k}(-1)^{t-k-1}\overline{m_t}=(-1)^{t-k}\overline{a_k}\cdot \overline{m_t}$.
\end{proof}

\begin{lem}\label{LemmaofUniquenessofJD}
Let $R$ be a commutative ring and $\underline{a}=a_1,\dots,a_r$ be an $R$-regular sequence. Suppose $B,~B'$ are two matrices with $s$ rows satisfying
\begin{equation}
(\underline{a}\cdot B)=(\underline{a}\cdot B'),
\end{equation}
then $(\underline{a}\cdot B,I_r(B))=(\underline{a}\cdot B',I_r(B'))$
\end{lem}
\begin{proof}
Let $\mathcal{L}=(\underline{a}\cdot B)=(\underline{a}\cdot B')$.

Let $E=(\underline{a})/(\underline{a}\cdot B)$ and consider the free presentation
\begin{equation*}
F_1\xrightarrow{[\delta~|~B]} F_0\rightarrow E\rightarrow 0
\end{equation*}
where $\delta$ represents first differential of the Koszul complex $\mathcal{K}$ of the $R$-regular sequence $\underline{a}$. Notice that 
\begin{equation}\label{FittingIdeal}
I_r([\delta~|~B])=\mathrm{Fitt}_0(E)=I_r([\delta~|~B'])
\end{equation}
as $(\underline{a}\cdot B)=(\underline{a}\cdot B')$ and the Fitting ideals do not depend on the presentation matrix.

Now $(\underline{a}\cdot [\delta~|~B])=(\underline{a}\cdot B)=\mathcal{L}$ because $\underline{a}\cdot \delta =0$. It suffices to show that 
\begin{equation}\label{equalityofkoszulandB}
\mathcal{L}+I_r([\delta~|~B])\subseteq \mathcal{L}+I_r(B)
\end{equation}
 as this would imply, using \eqref{FittingIdeal}, that \[\mathcal{L}+I_r(B)=\mathcal{L}+I_r([\delta~|~B])=\mathcal{L}+I_r([\delta~|~B'])=\mathcal{L}+I_r(B')\]

Now to prove \eqref{equalityofkoszulandB}, it is enough show to that $\overline{I_r([\delta~|~B])}\subseteq \overline{I_r(B)}$ in the ring $\overline{R}=R/\mathcal{L}$. Since $\delta$ is the first Koszul differential, we may assume the columns of $\delta$ are of the form $a_je_k-a_ke_j,~1\leq j,k\leq s$, where $\{e_j\}$ form a basis of $R^r$.

Now  any element of $I_r([\delta~|~B])$ involving a column of $\delta$ is of the form $\det [\delta'~|~M]$ where $M$ is a $s\times s-1$ submatrix of $[\delta~|~B]$ and $\delta'$ is a column of $\delta$. Then $\det [\delta'~|~M]$ is of the form 
\begin{equation}\label{equatioOnLinearCombinationOfMinors}
(-1)^{k+1}(a_jm_k-(-1)^{j-k}a_km_j)
\end{equation}
where $m_t$ is the determinant of the submatrix of $M$ obtained by removing the $t$-th row of $M$. Now in the ring $R/\mathcal{L}$, using Lemma \ref{LemmaofMinors}, we see that elements of the form \eqref{equatioOnLinearCombinationOfMinors} are zero. Thus $\overline{I_r([\delta~|~B])}\subseteq \overline{I_r(B)}$ in the ring $\overline{R}$ and hence $\mathcal{L}+I_r([\delta~|~B])= \mathcal{L}+I_r(B)$.
\end{proof}
	
Now using the lemma proved above, we show the uniqueness of the ideal of iterated Jacobian dual $\mathcal{L}+I_r(B_i(\phi))$.

\begin{thm}\label{Theorem of Uniqueness of Iterated Jacobian Dual}
	Let $R$ be a Noetherian ring and $\phi$ be a $m\times s$ matrix with entries in $R$. Suppose $I_1(\phi)\subseteq (a_1,\dots,a_r)$ and $a_1,\dots,a_r$ is a regular sequence.  Then the ideal $\mathcal{L}+I_r(B_i(\phi))$ of $R[T_1,\dots,T_m]$ is uniquely determined by the matrix $\phi$ and the regular sequence $a_1,\dots,a_r$.
\end{thm}
\begin{proof} Since the construction of the iterated Jacobian dual is inductive, we prove this result using the principle of mathematical induction. Suppose $B_1,B_2$ are two candidates for $B_1(\phi)=B(\phi)$. Using Lemma \ref{LemmaofUniquenessofJD}, we see that $\mathcal{L}+I_r(B_1(\phi))$ is a well defined ideal, proving the initial step of the induction hypothesis. Now suppose that $\mathcal{L}+I_r(B_j(\phi)),~1\leq j\leq i-1$ are well defined ideals. Now, if $B$, $B'$ are two matrices which satisfies \eqref{UniquenessofJacobianDual}, then we show that 
	\begin{equation*}
	\mathcal{L}+I_r(B)=\mathcal{L}+I_r(B')
	\end{equation*}
	We first notice, by Remark \ref{remarksIJD}, $\mathcal{L}+I_r(B)=\mathcal{L}_i+I_r(B)$ and $\mathcal{L}+I_r(B')=\mathcal{L}_i+I_r(B')$. So its enough to show $\mathcal{L}_i+I_r(B)=\mathcal{L}_i+I_r(B')$. Since $\mathcal{L}_i=(\underline{a}\cdot B)=(\underline{a}\cdot B')$, we now use Lemma \ref{LemmaofUniquenessofJD}, to show the result.
\end{proof}

Now, since $\mathcal{L}+I_r(B_i(\phi))\subseteq \mathcal{L}+I_r(B_{i+1}(\phi))$ and $R[T_1,\cdots,T_m]$ is Noetherian, the procedure stops after a certain number of iterations. Notice that, when $R$ is a polynomial ring and $\phi$ is linear, the procedure stops after the first iteration.

Using Cramer's Rule, we can see that $\mathcal{L}+I_r(B_1(\phi)) \subseteq (\mathcal{L}:(\underline{a}))$,  and hence \[\mathcal{L}+I_r(B_i(\phi)) \subseteq (\mathcal{L}:(\underline{a})^i)\] But its still unclear when the two ideals are equal or if their respective index of stabilizations are related.

\subsection{Ideals of Codimension two}

Let $R=k[x_1,\dots,x_d]$ be a polynomial ring and $I$ be a grade 2 perfect ideal satisfying the $G_d$ condition. Let $\varphi$ be the presentation matrix of $I$ and $\mu(I)=m>d$. If $\varphi$ is linear, then the defining ideal of the Rees algebra $\mathcal{R}(I)$ equals the \textit{expected form}. The expected form of the defining ideal of the Rees algebra is $\mathcal{L}+I_d(B(\varphi))$ (see \cite{MU}). When $\varphi$ is not linear, its interesting to study when the defining ideal of $\mathcal{R}(I)$ satisfy $\mathcal{A}=\mathcal{L}+I_d(B_i(\varphi))$. Such a form of the defining equations is easier to compute and has advantages, when computing the invariants such as relation type, regularity etc.

In \cite{M2}, a condition is given as to when the defining ideal of $\mathcal{R}(I)$ equals the expected form. An analogous condition is presented below for the ideal of iterated Jacobian dual.
\begin{rmk}
	Let $R=k[x_1,\dots x_d]$ be a polynomial ring with the homogeneous maximal ideal $\mathfrak{m}$ and $I$ be a grade 2 perfect ideal with presentation matrix $\varphi$. Assume $I$ satisfies the $G_d$ condition and let $I_1(\varphi)\subseteq (a_1,\dots,a_r)$, where $\underline{a}$ is a regular sequence. If $\hght (I_r(B_n(\varphi))+\mathfrak{m})/\mathfrak{m}R[T_1,\dots,T_m]\geq m-d$ and $\mathcal{L}+I_r(B_n(\varphi))$ is unmixed, then $\mathcal{A}=\mathcal{L}+I_r(B_n(\varphi))$
\end{rmk}
The proof of the above remark is identical to the one presented in \cite{M2}. The remark shows advantages in the feasibility of the bounds for \\$\hght (I_r(B_n(\varphi))+\mathfrak{m})/\mathfrak{m}R[T_1,\dots,T_m]$, but the unmixed condition is strong for it to be of practical use.

We now put efforts into finding a condition for the equality of the defining ideal of $\mathcal{R}(I)$ and the ideal of iterated Jacobian dual, under setting \ref{origset}. Looking at the generating set presented in \cite[3.6]{KPU1}, for $d=2$, it is clear that the defining ideal of $\mathcal{R}(I)$ and the ideal of iterated Jacobian duals are not always equal. A search for a condition, led us to Corollary \ref{condition of equality of ordinary and symbolic powers of K}.
\begin{rmkk}
Notice that when $\varphi$ is almost linear, $I_1(\varphi')\supset I_2(\varphi)\supset I_{m-d+1}(\varphi)$. Recall that $\varphi'$ is obtained from $\varphi$ by removing the last column. Thus the $G_d$ condition forces, $\gr I_{m-d+1}(\varphi)\geq d$ (\cite[6.6]{HSV2}), which shows  $I_1(\varphi)=(x_1,\dots,x_d)$.\end{rmkk}
\begin{thm}\label{Iterated Jacobian Dual in the almost linearly presented case}
	Let $A,K$ be as defined in Notation \ref{2ndnot}. Then in the setting of \ref{origset}, one has  $\frac{\overline{g}\overline{K}^n}{\overline{x_d}^n}\subseteq \overline{\mathcal{L}+I_d(B_n(\varphi))}$ in the ring $A$.
\end{thm}
\begin{proof}
	It is clear that $\frac{\overline{g}\overline{K}^i}{\overline{x_d}^i}\subseteq \frac{\overline{g}\overline{K}^{i+1}}{\overline{x_d}^{i+1}}$ and $\overline{\mathcal{L}+I_d(B_i(\varphi))}\subseteq \overline{\mathcal{L}+I_d(B_{i+1}(\varphi))}$. Write $D_i=\frac{\overline{g}\overline{K}^i}{\overline{x_d}^i}$ and $D_i'=\overline{\mathcal{L}+I_d(B_i(\varphi))}$.
	
	In Notation \ref{2ndnot}, we had defined $K=(\tilde{L}_1,\cdots,\tilde{L}_{m-2},I_{d-1}(B),x_d)$. Now let $B(\varphi)=(b_{ij}),~1\leq i\leq d,~1\leq j\leq m-1$. As in Notation \ref{2ndnot}, $B$ is a submatrix of $B(\varphi')$. 
	
	We prove the containment $D_i\subseteq D'_i,~1\leq i\leq n$, by induction. Suppose $i=1$. As $\overline{\tilde{L}_i}\in (\overline{x_d})$ in the ring $A$, it is clear that $\frac{\overline{g}\overline{\tilde{L}_i}}{\overline{x_d}}\in (\overline{g})\subseteq \overline{\mathcal{L}}$. Now let $w$ be a $d-1\times d-1$ minor of $B$. For ease of notation, assume that $w$ is the determinant of the submatrix of $B$ consisting of the first $d-1$ rows and the first $d-1$ columns of $B$. Consider $M$, a submatrix of $B(\varphi)$ consisting of the first $d$ rows and whose column indices belong to the set $\{1,\dots,d-1,m-1\}$. Using Cramer's rule, we have $\overline{g}\cdot \overline{w}=\overline{\mathrm{det}(M)}\cdot \overline{x_d}$ in the ring $A$. Thus we have $\frac{\overline{g}\overline{w}}{\overline{x_d}}=\overline{\det(M)}\in \overline{I_d(B_1(\varphi))}$ proving the initial step of induction.
		
	
	Now suppose that the result is true for $1\leq i<n$. Consider $\frac{\overline{g}\overline{w_1}\cdots \overline{w_n}}{\overline{x_d}^n}\in D_n$. We show that $\frac{\overline{g}\overline{w_1}\cdots \overline{w_n}}{\overline{x_d}^n}\in D'_n$.
	
	By induction hypothesis, we have $\frac{\overline{g}\overline{w_1}\cdots \overline{w_{n-1}}}{\overline{x_d}^{n-1}}=\overline{w'}\in D_{n-1}'$. Thus $\frac{\overline{g}\overline{w_1}\cdots \overline{w_n}}{\overline{x_d}^n}=\frac{\overline{w'}\overline{w_n}}{\overline{x_d}}$. If $w'\in\mathcal{L}$, then $\overline{w'}=0$ or $\overline{w'}=\overline{g}$ in the ring $A$. Thus by induction hypothesis, $\frac{\overline{w'}\overline{w_n}}{\overline{x_d}}\in D_1\subseteq D_1'\subseteq D_n'$. If $w'\in I_d(B_{n-1}(\varphi))$ and is purely in the $\underline{T}$-variables, then $w'\in I_d(B(\varphi'))$ and in this case,  $\frac{\overline{w'}\overline{w_n}}{\overline{x_d}}=0$ (recall that $A$ is a domain  and $n\geq 2$). 
	
	Therefore, assume $w'\in I_{d}(B_{n-1}(\varphi))\cap (x_1,\dots, x_d)=(u_{1},\dots,u_{l})$. Now it is enough to show that $\frac{\overline{u_p}\overline{w_n}}{\overline{x_d}}\in D'_n$, $1\leq p\leq l$. So, let $w'=u_{p}$ for some $p\in\{1,\cdots,l \}$
	
	Rewrite 
	\begin{equation}\label{w'decomp}
	w'=\sum\limits_{k=1}^dx_kw'_k \mbox{ for some }w'_k\in S.
	\end{equation}
	
	Now, if $\overline{w_n}\in (\overline{x_d})\subseteq K$, then $\frac{\overline{w'}\overline{w_n}}{\overline{x_d}}=\overline{w'}\in D'_{n-1}\subseteq D'_n$. Thus assume that $w_n\in I_{d-1}(B)$. Now  $$\frac{\overline{w'}\overline{w_n}}{\overline{x_d}}=\frac{\sum\limits_{k=1}^d\overline{x_k}\overline{w'_k}\overline{w_n}}{\overline{x_d}}.$$ For ease of notation assume that $w_n$ is the determinant of the submatrix consisting of the first $d-1$ rows and the first $d-1$ columns of $B$. Now let $M$ be a $d\times d-1$ submatrix consisting of the first $d$ rows and the first $d-1$ columns of $B(\phi)$. Hence in the ring $A$, using Lemma \ref{LemmaofMinors}, we have $\overline{x_k}\overline{w_n}=(-1)^{k-d}\overline{x_d}\overline{w_{n_k}}$ where $\overline{w_{n_k}}=\overline{\det M_k}$  and $M_k$ is the submatrix of $M$ obtained by removing the $k$-th row. Thus, 
	\begin{equation*}
	\frac{\overline{w'}\overline{w_n}}{\overline{x_d}}=\frac{\sum\limits_{k=1}^d(-1)^{k-d}\overline{x_d}\overline{w_{n_k}}\overline{w'_k}}{\overline{x_d}}=\sum_{k=1}\limits^d(-1)^{k-d}\overline{w_{n_k}}\overline{w'_k}
	\end{equation*}
	which is the determinant of the $d\times d$ matrix $[(b_{ij})~|~(w'_k)],~ 1\leq i\leq d,~1\leq j\leq d-1,~1\leq k\leq d$ an element of $I_d(B_n(\varphi))\subseteq D'_n$ (the non-unique decomposition in  \eqref{w'decomp} is taken care of by Theorem \ref{Theorem of Uniqueness of Iterated Jacobian Dual}).
\end{proof}
\begin{cor}\label{condition of equality of ordinary and symbolic powers of K}
	In the setting of Theorem \ref{2ndform}, if $\overline{K}^{(n)}=\overline{K}^n$, then the defining equation of the Rees algebra satisfy  $\mathcal{A}=\mathcal{L}+I_d(B_n(\varphi))$.
\end{cor}
Interestingly, the above corollary states that, under the conditions, $\mathcal{L}:(\underline{x})^n=\mathcal{L}+I_d(B_n(\varphi))$ and the index of stabilization of the ideal of iterated Jacobian dual is $n$.
\begin{rmk}
	In Theorem \ref{Iterated Jacobian Dual in the almost linearly presented case}, $D_1=D_1'$. To show the reverse inequality, notice that $x_d\in K$ and hence $\overline{g}=\frac{\overline{g}\overline{x_d}}{\overline{x_d}}\in D_1$ showing that $\overline{\mathcal{L}}\subseteq D_1$. Now let $w\in I_d(B_1(\varphi))$. Since $I_d(B(\varphi'))\subseteq J$, we can assume that $w\not\in I_d(B(\varphi'))$. Now in the ring $A$, $\overline{w}\overline{x_d}=\overline{g}\overline{w'}$. Thus $\overline{w}=\frac{\overline{g}\overline{w'}}{\overline{x_d}}\in D_1$.
	\\A natural question would be, if $D_i=D'_i,~1\leq i\leq n$ ?. The answer is affirmative, if a slight change is made while constructing the iterated Jacobian duals. The change being, for constructing $B_i(\varphi)$, instead of considering all the minors of $I_r(B_i(\varphi))\cap(\underline{x})$, we consider a subset of minors. These minors are determinants of  sub matrices all but one of whose columns are columns of $B(\varphi')$ except for the last column which is that of $B_{i-1}(\varphi)$. This type of construction has been independently studied by Cox,Hoffman and Wang, \cite{CHW} in the case of $d=2,~m=3$.
\end{rmk}
In the setting \ref{origset}, it was shown that $I_1(\varphi)=(x_1,\dots,x_d)$. But in general, the iterated Jacobian dual is defined to be constructed with any generating set containing $I_1(\varphi)$. The generating set need not even be homogeneous and this feature was explored in the case of $d=2$ by Hong, Simis and Vasconcelos in \cite{HoSV}.
We now present some examples on how to construct the iterated Jacobian duals.
\begin{ex}\label{exijd}
	Consider a matrix $\varphi=\left[\begin{array}{ccc}
	x_1& 0 & 0 \\ 
	x_2& x_1 &0  \\ 
	x_3& x_2 & x_1^2 \\ 
	0  & x_3 & x_3^2
	\end{array}  \right]$ in a polynomial ring $R=k[x_1,x_2,x_3]$. Since $\gr I_3(\varphi)\geq 2$, the converse of Hilbert-Burch Theorem, guarantees the existence of a grade 2 perfect ideal $I$ whose presentation matrix is $\varphi$. Also, the $G_3 $ condition is satisfied as $\gr I_1(\varphi)=3,~\gr I_2(\varphi)\geq 3$. Some candidates for the iterated Jacobian duals are as follows:
	\begin{equation*}
	B_1(\varphi)=\left[ \begin{array}{ccc}
	T_1& T_2 & xT_3 \\ 
	T_2& T_3 & 0 \\ 
	T_3& T_4 & zT_4
	\end{array} \right]
	\end{equation*}
	\begin{equation*}
	B_2(\varphi)=
	\left[\begin{array}{cccc}
	T_1& T_2 & xT_3 & T_3(T_3^2-T_2T_4) \\ 
	T_2& T_3 & 0 & 0 \\ 
	T_3& T_4 & zT_4 & T_4(-T_2^2+T_1T_3)
	\end{array}  \right]
	\end{equation*}
	In the next section we will show that the defining ideal of the Rees algebra $\mathcal{R}(I)$ satisfy $\mathcal{A}=\mathcal{L}+I_3(B_2(\varphi))$
\end{ex}
\begin{ex}
	Let $R=k[x_1,x_2]$. Let $I$ be a grade 2 perfect ideal whose presentation matrix  
	\begin{equation*}
	\varphi=\left[\begin{array}{ccc}
	x_1 & 0 & x_1^2\\
	x_2 & x_1 & x_2^2\\
	0 & x_2 &  x_1^2+x_2^2\\
	0 & 0 & x_1^2+x_2^2+x_1x_2
	\end{array}\right]
	\end{equation*}
	Some candidates for the iterated Jacobian duals are
	\begin{equation*}
	B_1(\varphi)=
	\begin{bmatrix}
	T_1 & T_2 & x_1T_1+x_1T_3+x_1T_4+x_2T_4\\
	T_2 & T_3 & x_2T_2+x_2T_3+x_2T_4
	\end{bmatrix}
	\end{equation*}
	
	\begin{equation*}
	\scalemath{0.76}{
		B_2(\varphi)=
		\begin{bmatrix}
		T_1 & T_2 & x_1T_1+x_1T_3+x_1T_4+x_2T_4\ & -T_1T_2-T_2T_3-T_2T_4 & -T_1T_3-T_3^3-T_3T_4\\
		T_2 & T_3 & x_2T_2+x_2T_3+x_2T_4 & T_1T_2+T_1T_3+T_1T_4-T_2T_4 & T_2^2+T_2T_3+T_2T_4-T_3T_4
		\end{bmatrix}}
	\end{equation*}
	\\Using \cite[3.6]{KPU1}, one can show that $f=T_2^2+T_1T_2+T_3^2+T_1T_3+T_3T_4+T_1T_4-T_2T_4\in \mathcal{A}$, but its clear that $f\not\in\mathcal{L}+I_2(B_2(\varphi))$. Subsequent iterations of the Jacobian dual, do not produce an element of bi-degree $(0,2)$. Thus $\mathcal{L}+I_2(B_2(\varphi))\neq\mathcal{A}$. 
	
\end{ex}

\section{Ideals with second analytic deviation one}
The aim of this section is to present a generating set of the defining ideal 
of the Rees algebra of ideals, whose second analytic deviation is one, in terms of the iterated 
Jacobian duals. Further, properties like depth, Cohen-Macaulayness, regularity of 
the Rees algebra are also studied.

The rest of this section assumes the setting of Theorem \ref{2ndform}. Let $\mathcal{F}(I)\cong \mathcal{R}(I)/(\underline{x})\mathcal{R}(I)$ be the \textit{special fiber} ring. The analytic spread, denoted by $\ell(I)$, is defined to be $\ell(I)=\dim \mathcal{F}(I)$. It is known that $\hght I\leq\ell(I)\leq \dim R$.

Further, we let $\mu(I)=d+1$. Since $I$ is of maximal analytic spread ($\ell(I)=d$, see for example \cite{UV1}), one has that the second analytic deviation $\mu (I)-\ell(I)$ is 1. Using Observation \ref{dimsym}, we also see that $A$ is a complete intersection domain. 

\begin{obs}\label{KisSCM}
Let $A,K$ be as defined in Notation \ref{2ndnot}. Then in the setting of \ref{origset}, $\overline{K}$ is generically a complete intersection and strongly Cohen-Macaulay ideal in $A$.
\end{obs}
\begin{proof}
Let $P$ be a prime ideal in the ring $A$ of height 1 containing $\overline{K}$. Since $\overline{(x_1,\dots,x_d)}$ is not an associated prime of $\overline{K}$ (Observation \ref{KisCM} and proof of Observation \ref{colon}), we have 
	$(\overline{x_d})_P:_{A_P}\overline{K}_P=\overline{(x_1\dots,x_d)}_P=A_P$ showing that $\overline{K}_P=(\overline{x_d})_P$. Thus $\overline{K}$ 
	is generically a complete intersection.
	
Notice that $\overline{K}=(\overline{w},\overline{x_d})$ where $I_{d-1}(B)=(w)$, is an almost complete intersection ideal of height 1, in the Cohen-Macaulay ring $A$. Also, $A/\overline{K}$ is 
	Cohen-Macaulay (Observation \ref{KisCM}) and hence  $\overline{K}$ is strongly Cohen-Macaulay (\cite[2.2]{HU1}).
\end{proof}
\begin{obs}\label{htofsubmaxminors}
Let $S,B(\phi')$ be as defined in Setting \ref{origset} and $\mu(I)=d+1$, then in the ring $S$, $\hght I_{d-1}B(\phi')=2$.
\end{obs}
\begin{proof}
We know that $A$ is a complete intersection domain (Observation \ref{dimsym}) of dimesion $d+2$. Since $\mu(I)=d+1$, we see that $I_d(B(\phi'))=0$. Thus $A= S/(L_1,\dots,L_{m-2})$. Notice that $A$ can be viewed as a symmetric algebra $A\cong \sym_{k[\underline{T}]}(\coker B(\phi'))$ over the ring $k[\underline{T}]$. Since $A$ is a domain we see that $$d+2=\dim A=\dim \sym_{k[\underline{T}]}(\coker B(\phi'))=\rk \coker B(\phi')+\dim k[\underline{T}].$$ Thus $\rk\coker B(\phi')=1$. Now using \cite[6.8,6.6]{HSV2} we see that $\gr I_{d-1}(B(\phi'))\geq 2$ which is the maximum possible bound.
\end{proof}
\begin{thm}\label{thmandev1}
	Let $R=k[x_1,\cdots, x_d]$ be a polynomial ring and let $I$ be a grade 2 perfect $R$-ideal whose presentation matrix $\varphi$ is almost linear. If $I$ satisfies $G_d$ and $\mu(I)=d+1$, then the defining ideal of $\mathcal{R}(I)$ satisfies
	\begin{equation*}
	\mathcal{A}=\mathcal{L}+I_d(B_n(\varphi))=\mathcal{L}:(x_1,\cdots, x_d)^n
	\end{equation*}
	where $n$ is the degree of the entries of the last column of $\varphi$. Furthermore, the special fiber ring $\mathcal{F}(I)\cong k[T_1,\dots,T_{d+1}]/(f)$ where $\deg f=n(d-1)+1$.
\end{thm}
\begin{proof}
	Using Corollary \ref{condition of equality of ordinary and symbolic powers of K} in the previous section, it suffices to show that 
	$\overline{K}^n=\overline{K}^{(n)}$.
	
	We now show that
	\begin{equation}\label{muKP}
	\mu(\overline{K}_P)\leq \hght P-1=1 \mbox{ for all }P\in V(\overline{K}),\mbox{~with~}\hght P=2.
	\end{equation}
	Let $P\in V(\overline{K})$ such that $\hght P=2$. If 
	$P\not\supset \overline{(x_1,\dots,x_d)}$, then, as above, $\overline{K}_P=(\overline{x_d})_P$ and hence \eqref{muKP} is trivially satisfied. 
	
	
	Now suppose $P\supset \overline{(x_1,\dots,x_d)}$. Observe that, since $\hght I_{d-1}(B(\varphi'))=2$ (Observation \ref{htofsubmaxminors}), we have $\hght (x_1,\dots,x_d,I_{d-1}(B(\varphi'))=d+2$ in 
	$k[\underline{x},\underline{T}]$, and hence  $\hght \overline{(x_1,\dots,x_d,I_{d-1}(B(\varphi'))}=3$ in 
	$A$. Thus $P\not\supset \overline{I_{d-1}(B(\varphi'))}$. Now let \[(w)=I_{d-1}(B)\subset I_{d-1}(B(\varphi '))=(w,w'_1,\cdots w'_{d-1}).\] Using Lemma \ref{LemmaofMinors}, we have $\overline{x_i}\cdot \overline{w}=(-1)^{t-k}\overline{x_d}\cdot \overline{w'_i}$. Since  $\overline{w}\in \overline{K}\subseteq P$, we have $\overline{w'_i}\not\in P$ for some $i
	\in \{1,\dots,d-1\}$. Thus $\overline{K}_P=(\overline{w})_P$ and hence \eqref{muKP} is satisfied.
	Putting all these observations together along with Observation \ref{KisSCM} we see that the hypothesis of \cite[3.4]{SV1} is satisfied. Thus $\overline{K}^n=\overline{K}^{(n)}$.
	
	The statement on the special fiber ring is clear as $(\underline{x})+\mathcal{A}=(\underline{x})+\mathcal{L}+I_d(B_n(\varphi))=(\underline{x})+(f')$. Also, $f'\in I_d(B_n(\varphi))\backslash I_d(B_{n-1}(\varphi))$ and hence $\deg f'=(0,n(d-1)+1)$. Now let $\overline{f'}=f$ where $\bar{ \ \ }$ represents the image in the ring $k[T_1,\dots, T_{d+1}]$.
\end{proof}
\begin{cor}
	Let $I$ be the ideal defining a set of 11 points in $\mathbb{P}^2$. Then for a general choice of points, the defining equations of the Rees algebra satisfies $\mathcal{A}=\mathcal{L}+I_3(B_2(\varphi))$, where $\varphi$ is a presentation matrix of $I$.
\end{cor}
\begin{proof}
	From the discussion in \cite[1.2]{T1}, we see that for a general choice of $11=\binom{4+1}{2}+1$ points, the presentation matrix $\varphi$ of $I$ is of size $4\times 3$ and satisfies all the hypothesis of the previous theorem. Also, the presentation matrix $\varphi$ is almost linear with the last column consisting of quadratic entries. Thus the defining ideal of the Rees algebra satisfies $\mathcal{A}=\mathcal{L}+I_3(B_2(\varphi))$.
\end{proof}
\begin{ex}
	In Example \ref{exijd}, $\overline{K}=(\overline{T_1T_3}-\overline{T_2^2},\overline{x_2})$, an almost complete intersection in the domain $A$. By the above theorem, $\mathcal{A}=\mathcal{L}+I_3(B_2(\varphi))=\mathcal{L}:(x_1,x_2,x_3)^2.$
\end{ex}

\begin{rmk}\label{remd1}
	We know that the defining ideal of the Rees algebra is also of the form $\mathcal{L}:(\underline{x})^n$. Since $\mathcal{L}+I_d(B(\varphi))\subset (\underline{x})$, the defining ideal of $\mathcal{R}(I)$ is not of the expected form. Also, $n$ is minimal by Remark \ref{nisMinimal}. Thus the Rees algebra is not a Cohen-Macaulay ring \cite[4.5]{SUV1}.
\end{rmk}
\vskip 2mm\noindent
\textbf{Depth, Relation type, Regularity}:
We first begin by constructing a series of short exact sequences which are instrumental in realizing invariants such as depth and regularity of the Rees algebra.

Let $\mathfrak{m}$ denote the ideal $(\underline{x})$ and $\mathfrak{n}$, the homogeneous maximal ideal of $A$. As in the above theorem, notice that $\overline{K}=(\overline{w},\overline{x_d}),~(w)= I_{d-1}(B)$. Also $\overline{K}$ is a Cohen-Macaulay ideal and $\mathfrak{m}A=(\overline{x_d}):\overline{K}$, which gives the exact sequence of bi-graded $A$-modules
\begin{equation}
0\rightarrow \mathfrak{m}A(0,-(d-1))\rightarrow A(-1,0)\oplus A(0,-(d-1))\rightarrow \overline{K}\rightarrow 0.
\end{equation}
Applying $\sym(~)$ to the above short exact sequence  and considering the $n$-th degree component, we obtain  
\begin{multline*}
\mathfrak{m}A(0,-(d-1))\otimes \sym_{n-1}(A(-1,0)\oplus A(0,-(d-1))) \xrightarrow{\sigma} \\\sym_n(A(-1,0)\oplus A(0,-(d-1)))\rightarrow \sym_n(\overline{K})\rightarrow 0.
\end{multline*}
Due to rank reasons $\ker\sigma$ is torsion, but the source of $\sigma$ is torsion-free module and hence $\sigma$ is injective. Thus we have an exact sequence 
\begin{multline}\label{bigradedsym}
0\rightarrow \mathfrak{m}A(0,-(d-1))\otimes \sym_{n-1}(A(-1,0)\oplus A(0,-(d-1))) \rightarrow \\\sym_n(A(-1,0)\oplus A(0,-(d-1)))\rightarrow \sym_n(\overline{K})\rightarrow 0.
\end{multline}
Using Observation \ref{KisSCM}, notice that $\overline{K}$ satisfies the $G_\infty$ condition and is strongly Cohen-Macaulay. Thus $\overline{K}$ is an $A$-ideal of linear type (\cite[2.6]{HSV1}). Therefore $\sym_n(\overline{K})\cong \overline{K}^n$.

Thus sequence \eqref{bigradedsym} now reads
\begin{equation}\label{regmAn}
0\rightarrow \bigoplus\limits_{i=0}^{n-1}\mathfrak{m}A(-i,-(n-i)(d-1))\rightarrow \bigoplus\limits_{i=0}^n A(-i,-(n-i)(d-1))\rightarrow \overline{K}^n\rightarrow 0.
\end{equation}

Recall that a Noetherian local ring $\mathcal{S}$ is said to be \textit{almost} Cohen-Macaulay when $\dep \mathcal{S}=\dim\mathcal{S}-1$.
\begin{thm}
	In the setting of Theorem \ref{thmandev1}, $\dep\mathcal{F}(I)=\dep \mathcal{R}(I)=d$, i.e the Rees algebra $\mathcal{R}(I)$ is almost Cohen-Macaulay and the special fiber ring $\mathcal{F}(I)$ is Cohen-Macaulay.
\end{thm}
\begin{proof} From the short exact sequence, 
	\begin{equation}\label{regmA}
	0\rightarrow \mathfrak{m}A\rightarrow A\rightarrow A/\mathfrak{m}A\rightarrow 0
	\end{equation}
	we have $\dep \mathfrak{m}A=d+2$. Now from \eqref{regmAn} we have $\dep \overline{K}^n\geq d+1$. The sequence
	\begin{equation*}
	0\rightarrow\overline{\mathcal{A}}\rightarrow A\rightarrow \mathcal{R}(I)\rightarrow 0
	\end{equation*}
	 and the isomorphism $\overline{\mathcal{A}}=\frac{\overline{g}\overline{K}^{(n)}}{\overline{x_d}^n}\cong K^{(n)}$ now implies that $\dep\mathcal{R}(I)\geq d$. Since $\mathcal{R}(I)$ is not Cohen-Macaulay (Remark \ref{remd1}), we have $\dep\mathcal{R}(I)=d$. 
	
	Since $\mathcal{F}(I)\cong k[T_1,\dots,T_{d+1}]/(f)$ (Theorem \ref{thmandev1}), we have $\dep \mathcal{F}(I)=d$.
\end{proof}
We now define two important invariants namely relation type and regularity of the Rees algebra. The relation type $\rt (I)$ is defined to be the maximum $\underline{T}$-degree appearing in a homogeneous minimal generating set of the defining ideal of the Rees algebra.

For regularity, we use the definition as in \cite{Tr1}. Let $\mathcal{S}=\bigoplus_{n\geq 0}\mathcal{S}_n$ be a finitely generated standard graded ring over a Noetherian commutative ring $\mathcal{S}_0$. For any graded $\mathcal{S}$-module $M$ we denote by $M_n$, the homogeneous part of degree $n$ of $M$, and define
\begin{equation*}
a(M):= \begin{cases}
\mathrm{max}\{n~:~M_n\neq 0\} & \mbox{if~}M\neq 0 \\ 
-\infty & \mbox{if~}M= 0 
\end{cases} 
\end{equation*}
Let $\mathcal{S}_+$ be the ideal generated by the homogeneous elements of positive degree of $\mathcal{S}$. For $i\geq 0$, set $a_i(\mathcal{S}):=a(H_{\mathcal{S}_+}^i(\mathcal{S}))$, where $H_{\mathcal{S}_+}^i(.)$ denotes the $i$the local cohomology functor with respect to the ideal $\mathcal{S}_+$. The \textit{Castelnuovo-Mumford} regularity of $S$ defined as the number
\begin{equation*}
\reg~\mathcal{S}:=\mathrm{max}\{a_i(\mathcal{S})+i~:~i\geq 0\}
\end{equation*}
We also make use of Castelnuovo-Mumford regularity on short exact sequences as given in \cite[20.19]{E1}.

We compute the regularity of $\mathcal{R}(I)$ with respect to $\mathfrak{m},\mathfrak{n}$ and $(\underline{T})$. This means when computing the $\reg_\mathfrak{m}\mathcal{R}(I)$ we set $\deg x_i=1,~\deg T_i=0$. Analogously the grading scheme is set for $\reg_\mathfrak{n}\mathcal{R}(I)$ and $\reg_{(\underline{T})}\mathcal{R}(I)$.
\begin{thm}
	In the setting of Theorem \ref{thmandev1}, 
	\begin{equation*}
	\rt (I)=\reg ~\mathcal{F}(I)+1=\reg_{(\underline{T})}\mathcal{R}(I)+1=n(d-1)+1
	\end{equation*}
	Furthermore, $\reg_\mathfrak{m}\mathcal{R}(I)\leq n-1$ and $\reg_\mathfrak{n}\mathcal{R}(I)\leq (n+1)(d-1)$
\end{thm}
\begin{proof}
	Since $\overline{\mathcal{A}}=\frac{\overline{g}\overline{K}^{(n)}}{\overline{x_d}^n}=\frac{\overline{g}\overline{K}^n}{\overline{x_d}^n}$, the relation type, $\rt(I)$, is easily computed by considering the generating set of $\overline{K}^n$. Thus $\rt(I)=n(d-1)+1$.
	
	Using Theorem \ref{thmandev1}, we have $\mathcal{F}(I)\cong k[T_1,\dots,T_{d+1}]/(f)$ where $\deg f=n(d-1)+1$. Thus $\reg~\mathcal{F}(I)=n(d-1)$.

	 Also, $\rt(I)-1\leq \reg_{(\underline{T})}\mathcal{R}(I)$. Therefore, in order to show the equality $\reg_{(\underline{T})}\mathcal{R}(I)+1=n(d-1)+1$, its enough to show that $\reg_{(\underline{T})}\mathcal{R}(I)\leq n(d-1)$.
	
	To compute the regularity we make use of  exact sequences \eqref{regmAn} and \eqref{regmA}. Notice that $A/\mathfrak{m}A\cong k[T_1,\dots,T_{d+1}]$. Since $A$ is a complete intersection domain defined by forms which are linear in both $\underline{x}$ and $\underline{T}$ variables, we have 
	\begin{align*}
	\reg_{(\underline{T})} A=\reg_{(\underline{T})} A/\mathfrak{m}A=0\\
\reg_\mathfrak{m}A=\reg_\mathfrak{m}A/\mathfrak{m}A=0\\
	\reg _\mathfrak{n} A=d-1,~\reg _\mathfrak{n} A/\mathfrak{m}A=0\
	\end{align*}
	Thus from \eqref{regmA} we have,
	\begin{equation*}
	\reg_{(\underline{T})}\mathfrak{m}A\leq 1,~\reg_\mathfrak{m}\mathfrak{m}A\leq 1,~\reg_\mathfrak{n}\mathfrak{m}A=d-1.
	\end{equation*}
	Let $M=\bigoplus\limits_{i=0}^{n-1}\mathfrak{m}A(-i,-(n-i)(d-1))$ and $N=\bigoplus\limits_{i=0}^n A(-i,-(n-i)(d-1))$. Now 
	\begin{align*}
	 \reg_{(\underline{T})}M&\leq n(d-1)+1 & \reg_{(\underline{T})}N&= n(d-1)
	 \\
	 \reg_\mathfrak{m}M&\leq n &\reg_{\mathfrak{m}}N&=n
	 \\
	 \reg_{\mathfrak{n}}M&=(n+1)(d-1) & \reg_\mathfrak{n}N&=(n+1)(d-1).
	\end{align*}
	Now using \eqref{regmAn} we have 
	\begin{align*}
	\reg_{(\underline{T})}\overline{K}^n&\leq n(d-1)\\\numberthis\label{regKn}
	\reg_\mathfrak{m}\overline{K}^n&\leq n\\
	\reg_\mathfrak{n}\overline{K}^n&\leq(n+1)(d-1).
	\end{align*}
	Next, consider the short exact sequence
	\begin{equation*}
	0\rightarrow \overline{\mathcal{A}}\rightarrow A\rightarrow \mathcal{R}(I)\rightarrow 0.
	\end{equation*}
	We now use the bigraded isomorphism $\overline{\mathcal{A}}\cong \overline{K}^n(0,-1)$ and the inequalities in \eqref{regKn} to show
	\begin{align*}
	\reg_{(\underline{T})}\mathcal{R}(I)&\leq n(d-1)
		\\\reg_\mathfrak{m}\mathcal{R}(I)&\leq n-1
	\\\reg_\mathfrak{n}\mathcal{R}(I)&\leq (n+1)(d-1).
	\end{align*}

\end{proof}

\bibliography{mainbib}{}

\providecommand{\bysame}{\leavevmode\hbox to3em{\hrulefill}\thinspace}
\providecommand{\MR}{\relax\ifhmode\unskip\space\fi MR }
\providecommand{\MRhref}[2]{%
  \href{http://www.ams.org/mathscinet-getitem?mr=#1}{#2}
}
\providecommand{\href}[2]{#2}
\begin{thebibliography}{10}

\bibitem{BKM}
W.~Bruns, A.~R. Kustin, and M.~Miller, \emph{The resolution of the generic
  residual intersection of a complete intersection}, J. Algebra \textbf{128}
  (1990), 214--239.

\bibitem{BC1}
L.~Bus{\'e} and M.~Chardin, \emph{Implicitizing rational hypersurfaces using
  approximation complexes}, J. Symbolic Comput. \textbf{40} (2005), 1150--1168.

\bibitem{BCS}
L.~Bus{\'e}, M.~Chardin, and A.~Simis, \emph{Elimination and nonlinear
  equations of {R}ees algebras}, J. Algebra \textbf{324} (2010), 1314--1333,

\bibitem{BJ1}
L.~Bus{\'e} and J.-P. Jouanolou, \emph{On the closed image of a rational map
  and the implicitization problem}, J. Algebra \textbf{265} (2003), 312--357.

\bibitem{CHW}
D.~Cox, J.~W. Hoffman, and H.~Wang, \emph{Syzygies and the {R}ees algebra}, J.
  Pure Appl. Algebra \textbf{212} (2008), 1787--1796.

\bibitem{CD}
C.~D'Andrea, \emph{Resultants and moving surfaces}, J. Symbolic Comput.
  \textbf{31} (2001), 585--602.

\bibitem{Ea}
J.~A. Eagon, \emph{I{DEALS} {GENERATED} {BY} {THE} {SUBDETERMINANTS} {OF} {A}
  {MATRIX}}, ProQuest LLC, Ann Arbor, MI, 1961, Thesis (Ph.D.)--The University
  of Chicago.

\bibitem{E1}
D.~Eisenbud, \emph{Commutative algebra}, Graduate Texts in Mathematics, vol.
  150, Springer-Verlag, New York, 1995, With a view toward algebraic geometry.

\bibitem{T1}
H.~T. H{\`a}, \emph{On the {R}ees algebra of certain codimension two perfect
  ideals}, Manuscripta Math. \textbf{107} (2002), 479--501.

\bibitem{HSV1}
J.~Herzog, A.~Simis, and W.~V. Vasconcelos, \emph{Approximation complexes of
  blowing-up rings}, J. Algebra \textbf{74} (1982), 466--493.

\bibitem{HSV2}
\bysame, \emph{Approximation complexes of blowing-up rings. {II}}, J. Algebra
  \textbf{82} (1983), 53--83.

\bibitem{HoSV}
J.~Hong, A.~Simis, and W.~V. Vasconcelos, \emph{On the homology of
  two-dimensional elimination}, J. Symbolic Comput. \textbf{43} (2008),
  275--292.

\bibitem{HU1}
C.~Huneke, \emph{Strongly {C}ohen-{M}acaulay schemes and residual
  intersections}, Trans. Amer. Math. Soc. \textbf{277} (1983), 739--763.

\bibitem{HR}
C.~Huneke and M.~Rossi, \emph{The dimension and components of symmetric
  algebras}, J. Algebra \textbf{98} (1986), 200--210.

\bibitem{HU2}
C.~Huneke and B.~Ulrich, \emph{Residual intersections}, J. Reine Angew. Math.
  \textbf{390} (1988), 1--20.

\bibitem{HU3}
\bysame, \emph{Generic residual intersections}, Commutative algebra
  ({S}alvador, 1988), Lecture Notes in Math., vol. 1430, Springer, Berlin,
  1990, pp.~47--60.

\bibitem{MJ1}
M.~R. Johnson, \emph{Second analytic deviation one ideals and their {R}ees
  algebras}, J. Pure Appl. Algebra \textbf{119} (1997), 171--183.

\bibitem{KAP}
I.~Kaplansky, \emph{Commutative rings}, revised ed., The University of Chicago
  Press, Chicago, Ill.-London, 1974.

\bibitem{KPU1}
A.~R. Kustin, Claudia Polini, and B.~Ulrich, \emph{Rational normal scrolls and
  the defining equations of {R}ees algebras}, J. Reine Angew. Math.
  \textbf{650} (2011), 23--65.

\bibitem{LN1}
Nguyen P.~H. Lan, \emph{On {R}ees algebras of linearly presented ideals}, J.
  Algebra \textbf{420} (2014), 186--200.

\bibitem{LP1}
K.-N. Lin and C.~Polini, \emph{Rees algebras of truncations of complete
  intersections}, J. Algebra \textbf{410} (2014), 36--52.

\bibitem{M2}
S.~Morey, \emph{Equations of blowups of ideals of codimension two and three},
  J. Pure Appl. Algebra \textbf{109} (1996), 197--211.

\bibitem{MU}
S.~Morey and B.~Ulrich, \emph{Rees algebras of ideals with low codimension},
  Proc. Amer. Math. Soc. \textbf{124} (1996), 3653--3661.

\bibitem{N1}
D.~G. Northcott, \emph{A homological investigation of a certain residual
  ideal}, Math. Ann. \textbf{150} (1963), 99--110.

\bibitem{SUV1}
A.~Simis, B.~Ulrich, and W.~V. Vasconcelos, \emph{Cohen-{M}acaulay {R}ees
  algebras and degrees of polynomial relations}, Math. Ann. \textbf{301}
  (1995), 421--444.

\bibitem{SV1}
A.~Simis and W.~V. Vasconcelos, \emph{The syzygies of the conormal module},
  Amer. J. Math. \textbf{103} (1981), 203--224.

\bibitem{Tr1}
N.~V. Trung, \emph{The {C}astelnuovo regularity of the {R}ees algebra and the
  associated graded ring}, Trans. Amer. Math. Soc. \textbf{350} (1998),
  2813--2832.

\bibitem{UV1}
B.~Ulrich and W.~V. Vasconcelos, \emph{The equations of {R}ees algebras of
  ideals with linear presentation}, Math. Z. \textbf{214} (1993), 79--92.

\bibitem{Vas1}
W.~V. Vasconcelos, \emph{On the equations of {R}ees algebras}, J. Reine Angew.
  Math. \textbf{418} (1991), 189--218.

\end{thebibliography}
\bibliographystyle{amsplain}

\end{document}